\documentclass[pdflatex,sn-mathphys-num]{sn-jnl}

\usepackage[T1]{fontenc}
\usepackage[utf8]{inputenc}
\usepackage{amsmath,amssymb,amsfonts,amsthm,mathtools}
\usepackage{bm}
\usepackage{graphicx}
\usepackage{subcaption} 
\usepackage{booktabs}
\usepackage{multirow}
\usepackage{enumitem}
\usepackage{xcolor}
\usepackage{hyperref}
\usepackage{array}
\usepackage{caption}
\usepackage{setspace}
\usepackage{comment}
\hypersetup{colorlinks=true,linkcolor=blue,citecolor=blue,urlcolor=blue}

\usepackage{mathrsfs}%
\usepackage{hyperref}
\usepackage[title]{appendix}%
\usepackage{textcomp}%
\usepackage{manyfoot}%
\usepackage{algorithm}%
\usepackage{listings}%

\theoremstyle{thmstyleone}%
\newtheorem{theorem}{Theorem}
\newtheorem{assumption}{Assumption}
\newtheorem{lemma}{Lemma}

\theoremstyle{thmstyletwo}%
\newtheorem{remark}{Remark}%

\theoremstyle{thmstylethree}%

\begin{document}

\title[Penalty-scaling effects in NIPG of viscous rotating SWEs]{Penalty-scaling effects in nonsymmetric interior-penalty DG discretizations of viscous rotating shallow-water equations}


\author[1]{\fnm{Xue} \sur{Zhang}}\email{zhangxue\_1998@nudt.edu.cn}

\author*[2]{\fnm{Jingmin} \sur{Xia}}\email{jingmin.xia@nudt.edu.cn}

\author[1]{\fnm{Xu} \sur{Qian}}\email{qianxu@nudt.edu.cn}

\affil[1]{\orgdiv{College of Science}, \orgname{National University of Defense Technology}, \orgaddress{\street{No.~109 Deya Road}, \city{Changsha}, \postcode{410073}, \state{Hunan}, \country{China}}}

\affil[2]{\orgdiv{College of Meteorology and Oceanography}, \orgname{National University of Defense Technology}, \orgaddress{\street{No.~109 Deya Road}, \city{Changsha}, \postcode{410073}, \state{Hunan}, \country{China}}}


\abstract{
We investigate how the scaling of the interior-penalty parameter affects nonsymmetric interior-penalty Galerkin (NIPG) discretizations of the viscous rotating shallow-water equations in geopotential variables.
The hyperbolic terms are approximated by a local Lax--Friedrichs flux, while viscosity acts on the momentum variables through a penalty law $\mu_e=\sigma h_e^{-\beta}$.
The standard choice $\beta=1$ and the super-penalized choice $\beta=3$ are compared with a symmetric interior-penalty Galerkin reference.
For the diffusion form, we establish consistency, continuity for $\beta\ge 1$, and an exact coercivity identity in the momentum DG seminorm.
Manufactured-solution tests show that super-penalization can recover the expected momentum $L^2$ accuracy, whereas the coupled geopotential variable need not exhibit the same improvement.
Rotating and topography-aware tests further show that the standard scaling generally gives the better accuracy-cost compromise for the explicit implementation considered here.
}



\maketitle

\section{Introduction}

The rotating shallow-water equations (RSWEs) are a standard model for large-scale geophysical flows \cite{Pedlosky2013,Williamson1992} and can develop complex
flow structures such as barotropic instability \cite{Galewsky2004}. 
Their numerical approximation must combine accurate transport, free-surface evolution, Coriolis coupling, and, in the viscous case considered here, 
a parabolic momentum operator. The resulting hyperbolic-parabolic structure provides a useful setting in which to examine not only convergence, 
but also the computational consequences of stabilization parameters.

Discontinuous Galerkin (DG) methods combine elementwise conservation, geometric flexibility, and high-order approximation with local interelement 
coupling \cite{Cockburn1990,ref_article3}. They have therefore been used extensively for shallow-water and related geophysical systems 
\cite{ref_article6,ref_article3,ref_article5,ref_article1}. Once viscous momentum diffusion is included, however, the treatment of the 
elliptic operator becomes a central design choice \cite{Baumann1999,Liu2019}.

Within the interior-penalty family, symmetric (SIPG), incomplete (IIPG), and nonsymmetric (NIPG) variants differ in adjoint consistency, coercivity 
requirements, and attainable $L^2$ accuracy \cite{Baumann1999, Epshteyn2007, Hesthaven2007}. NIPG has the useful property that its diffusion 
form is coercive for every positive penalty prefactor \cite{ref_article7,Houston2005,Riviere2001,Hesthaven2007}. Its lack of adjoint consistency, on the other hand, generally yields a suboptimal $L^2$ 
estimate under the standard $h^{-1}$ scaling. Optimal rates are sometimes observed for odd polynomial degrees on structured meshes, but this behavior 
is not a general mesh-independent guarantee \cite{Larson2004,Houston2002}. Super-penalty techniques strengthen the jump term and can restore optimal 
elliptic $L^2$ approximation under additional assumptions \cite{Chen2006,Gudi2009}.
More recent studies have mainly focused on elliptic 
and diffusion-dominated problems, including porous media flow in fractured media \cite{Liu2026} and convection-diffusion problems with sharp layers 
\cite{Zhang2024}, where the primary focus is on stability and accuracy. However, the accuracy-stiffness trade-off produced by different NIPG penalty 
scalings has received substantially less attention for coupled, time-dependent geophysical flow systems.

Motivated by this, the present work studies a DG discretization of the viscous RSWE in geopotential variables. The hyperbolic part is discretized using a local Lax--Friedrichs flux, while the viscous operator is treated using a NIPG formulation. The study focuses on how the penalty scaling law
\[
\mu_e=\sigma h_e^{-\beta}
\]
influences the performance of the scheme, in particular by comparing the standard choice ($\beta=1$) with the super-penalized scaling ($\beta=3$).
This problem lies at the intersection of three well-established research directions. The first is the DG literature for shallow water and related geophysical flow systems \cite{ref_article6, ref_article3,ref_article5, ref_article1}. The second is the theory and practice of interior-penalty DG discretizations for elliptic and parabolic operators, including SIPG and NIPG formulations \cite{ref_article7,ref_article2,ref_article4}. The third is the broader scientific-computing question of how method parameters, rather than only approximation spaces or flux choices, determine the real operating range of a discretization.
The present work therefore isolates the role of the penalty exponent and investigates its impact on the NIPG scheme. In addition, 
a comparison with the SIPG method is provided as a benchmark reference in convergence tests for rotating shallow water problems.

The main contributions are:
\begin{enumerate}[leftmargin=2em]
\item a conservative DG formulation in geopotential variables with a penalty-scaled NIPG momentum operator;
\item consistency, continuity, and a coercivity identity for the diffusion form, stated explicitly in the momentum DG seminorm;
\item a conditional stability and error framework for the linearized coupled problem that separates proven diffusion properties from assumptions on the hyperbolic linearization;
\item numerical evidence from manufactured, rotating, and topography-aware tests that distinguishes spatial accuracy, interface-jump control, and explicit computational stiffness.
\end{enumerate}

The remainder of the paper is organized as follows. Section~\ref{sec:model} introduces the governing equations. Section~\ref{sec:space} presents the DG formulation and the penalty scaling. Section~\ref{sec:analysis} gives the linearized analytical framework. Section~\ref{sec:time} discusses explicit time integration and the penalty-aware stability restriction. Section~\ref{sec:numerics} presents the numerical experiments, and Section~\ref{sec:conclusion} gives the conclusions.

\section{Viscous rotating shallow-water equations in geopotential variables}
\label{sec:model}

Let $\Omega\subset\mathbb{R}^2$ be a bounded polygonal domain and let $T>0$. The free-surface displacement is denoted by $\eta(\mathbf{x},t)$, 
the prescribed bottom field by $b(\mathbf{x})\in W^{1,\infty}(\Omega)$, and the total depth by
\[
H=\eta+b,\qquad H(\mathbf{x},t)\ge H_{\min}>0.
\]
With the sign convention used in this paper, the viscous RSWE reads
\begin{equation}
\label{eq1}
\begin{aligned}
\partial_t\mathbf{u}+\mathbf{u}\cdot\nabla\mathbf{u}
+f_c\mathbf{k}\times\mathbf{u}+g\nabla\eta
-\frac{\nu_T}{H}\Delta(H\mathbf{u})&=0,\\
\partial_tH+\nabla\cdot(H\mathbf{u})&=0,
\end{aligned}
\qquad \text{in }\Omega\times(0,T],
\end{equation}
where $\mathbf{u}=(u,v)^\top$ is the depth-averaged horizontal velocity in the $x$- and $y$-directions, $f_c$ is the Coriolis parameter, 
$\mathbf{k}$ is the local vertical unit vector, $g$ is gravitational acceleration, and $\nu_T>0$ is the momentum viscosity.
It is also noted that by using vertical integration and the hydrostatic-pressure assumption, the incompressible Navier--Stokes equations can be reduced to the shallow water system \cite{Aizinger2002}.

The numerical experiments use periodic boundaries. For the general boundary notation below, a compatible exterior state
\[
\mathbf{q}_{\mathrm D}=(\phi_{\eta,\mathrm D},U_{\mathrm D},V_{\mathrm D})^\top
\]
is prescribed on $\Gamma_{\mathrm D}$. Its components correspond to the prescribed velocity and depth through $\phi_{\eta,\mathrm{D}}=g (H_{\mathrm{D}}-b)$, $U_{\mathrm D}=gH_{\mathrm D}u_{\mathrm D}$ and $V_{\mathrm D}=gH_{\mathrm D}v_{\mathrm D}$. Only its momentum components enter the viscous Dirichlet data, whereas the full state is used by the hyperbolic numerical flux. This distinction avoids applying the nonlinear flux functions to velocity data alone.
The compatible initial data $H_0 \in L^2(\Omega),\ \mathbf{u}_0 \in [L^2(\Omega)]^2$ is also given
\begin{equation*}
\begin{aligned}
 H|_{t=0}=H_0,\qquad
\mathbf{u}|_{t=0}=\mathbf{u}_0 \quad \text{in }\Omega.
\end{aligned}
\end{equation*}

For flux construction and DG discretization it is convenient to recast \eqref{eq1} in conservative geopotential variables. 
Multiplying the momentum equation by $H$ and using the continuity equation gives
\[
H(\partial_t\mathbf{u}+\mathbf{u}\cdot\nabla\mathbf{u})
=\partial_t(H\mathbf{u})+\nabla\cdot(H\mathbf{u}\otimes\mathbf{u}).
\]
Then we obtain
\begin{equation}
\label{eq2}
\begin{aligned}
\partial_t(gH\mathbf{u})+\nabla\cdot(gH\mathbf{u}\otimes\mathbf{u})
+f_c\mathbf{k}\times(gH\mathbf{u})+gH\nabla(g\eta)
-\nu_T\Delta(gH\mathbf{u})&=0,\\
\partial_t(gH)+\nabla\cdot(gH\mathbf{u})&=0.
\end{aligned}
\end{equation}

Introducing
\[
\phi_\eta=g\eta,\qquad \phi_b=gb,\qquad \phi=gH=\phi_\eta+\phi_b,
\qquad U=gHu,\qquad V=gHv,
\]
the conservative state $\mathbf q=(\phi_\eta,U,V)^\top$ satisfies
\begin{equation}
\label{eq3}
\partial_t\mathbf q+\nabla\cdot\mathcal F(\mathbf q)-\Delta S(\mathbf q)=R(\mathbf q),
\end{equation}
with
\[
\mathcal F(\mathbf q)=
\begin{pmatrix}
U & V\\[0.2em]
\dfrac{U^2}{\phi}+\dfrac12(\phi^2-\phi_b^2) & \dfrac{UV}{\phi}\\[0.8em]
\dfrac{UV}{\phi} & \dfrac{V^2}{\phi}+\dfrac12(\phi^2-\phi_b^2)
\end{pmatrix},
\quad
S(\mathbf q)=
\begin{pmatrix}0\\ \nu_TU\\ \nu_TV\end{pmatrix},
\]
\[
R(\mathbf q)=
\begin{pmatrix}
0\\[0.2em]
f_cV+\phi_\eta\,\partial_x\phi_b\\[0.2em]
-f_cU+\phi_\eta\,\partial_y\phi_b
\end{pmatrix}.
\]
This provides the model basis for the penalty-dependent analysis and computations developed below. 

\section{DG discretization and penalty design}
\label{sec:space}

Let $Q_h$ be a shape-regular triangulation of $\Omega$, let $\mathcal E_h$ be its facets, and write $\mathcal E_h=\Gamma_h\cup\Gamma_{\mathrm D}$ for 
interior and Dirichlet facets. Periodic pairs are treated as interior facets. For polynomial degree $k\ge1$, define
\[
V_k=\{v\in L^2(\Omega):v|_K\in\mathbb P_k(K)\ \text{for every }K\in Q_h\},
\qquad \mathbf V_h=V_k^3.
\]
For $s\ge0$, the broken Sobolev norm is denoted by
\[
\|v\|_{H^s(Q_h)}^2=\sum_{K\in Q_h}\|v\|_{H^s(K)}^2.
\]
On an interior facet $e=\partial K^+\cap\partial K^-$, a fixed normal $\mathbf n=(n_x,n_y)^\top$ points from $K^+$ to $K^-$. For scalar or vector traces, averages and jumps are taken componentwise:
\[
\{\!\{z\}\!\}=\tfrac12(z^++z^-),\qquad [[z]]=z^+-z^-.
\]
On $\Gamma_{\mathrm D}$, the interior trace is denoted by $z^+$ and the prescribed exterior trace by $z_{\mathrm D}$.
For a test function $\mathbf w_h$, the exterior trace is zero, so $[[\mathbf w_h]]=\mathbf w_h^+$ and $\{\!\{\nabla\mathbf w_h\}\!\}=\nabla\mathbf w_h^+$ on $\Gamma_{\mathrm D}$.

The semidiscrete problem is: find $\mathbf q_h(t)\in\mathbf V_h$ such that
\begin{equation}
\label{eq4}
(\partial_t\mathbf q_h,\mathbf w_h)_\Omega
+b_h(\mathbf q_h,\mathbf w_h)+a_h(\mathbf q_h,\mathbf w_h)
=L_h(\mathbf q_h,\mathbf w_h),
\qquad\forall\mathbf w_h\in\mathbf V_h.
\end{equation}
Here, $(\cdot, \cdot)_{\Omega}$ denotes the standard $L^2(\Omega)$ inner product.
The source and nonhomogeneous viscous boundary terms are
\begin{equation}
\label{eq5}
\begin{aligned}
L_h(\mathbf q_h,\mathbf w_h)
={}&\sum_{K\in Q_h}\int_K R(\mathbf q_h)\cdot\mathbf w_h\,\mathrm d\mathbf x\\
&+\sum_{e\subset\Gamma_{\mathrm D}}\int_e
\bigl(\nabla\mathbf w_h\,\mathbf n+\mu_e\mathbf w_h\bigr)\cdot S(\mathbf q_{\mathrm D})\,\mathrm ds.
\end{aligned}
\end{equation}
It is noted that for periodic calculations, the second line is actually absent.
The handling of the hyperbolic and momentum terms is described below.

\subsection{Hyperbolic flux}

The hyperbolic form $b_h$ in \eqref{eq4} is
\begin{equation}
\label{eq6}
b_h(\mathbf q_h,\mathbf w_h)
=-\sum_{K\in Q_h}\int_K\mathcal F(\mathbf q_h):\nabla\mathbf w_h\,\mathrm d\mathbf x
+\sum_{e\in\mathcal E_h}\int_e\widehat{\mathcal F}_e\cdot[[\mathbf w_h]]\,\mathrm ds.
\end{equation}
On an interior or periodic facet $e \in \mathcal{E}_h$, we use the local Lax--Friedrichs numerical flux \cite{Toro1992}
\[
\widehat{\mathcal F}_e
=\{\!\{\mathcal F(\mathbf q_h)\}\!\}\mathbf n
+\frac{\tau_e}{2}[[\mathbf q_h]],
\qquad
\tau_e=\max_{\pm}\bigl(|\mathbf u^\pm\cdot\mathbf n|+\sqrt{gH^\pm}\bigr).
\]
This choice provides a robust and standard hyperbolic treatment, which is useful here because it limits the number of moving parts in the penalty comparison.
On $\Gamma_{\mathrm D}$, the same local Lax--Friedrichs formula is used with the pair $(\mathbf q_h^+,\mathbf q_{\mathrm D})$. 

\subsection{NIPG treatment of momentum diffusion}

Since $S(\mathbf q)=\mathbf B\mathbf q$ with $\mathbf B=\operatorname{diag}(0,\nu_T,\nu_T)$, the NIPG form \cite{ref_article2} is
\begin{equation}
\label{eq7}
\begin{aligned}
a_h(\mathbf z,\mathbf w)
={}&\sum_{K\in Q_h}\int_K\nabla S(\mathbf z):\nabla\mathbf w\,\mathrm d\mathbf x
-\sum_{e\in\mathcal E_h}\int_e\{\!\{\nabla S(\mathbf z)\}\!\}\mathbf n\cdot[[\mathbf w]]\,\mathrm ds\\
&+\sum_{e\in\mathcal E_h}\int_e\{\!\{\nabla\mathbf w\}\!\}\mathbf n\cdot[[S(\mathbf z)]]\,\mathrm ds
+\sum_{e\in\mathcal E_h}\int_e\mu_e[[S(\mathbf z)]]\cdot[[\mathbf w]]\,\mathrm ds,
\end{aligned}
\end{equation}
where boundary jumps in the right-hand side use the interior trace; the prescribed trace is accounted for by \eqref{eq5}. The penalty law is
\begin{equation}
\label{eq:mu}
\mu_e=\sigma h_e^{-\beta},\qquad \sigma>0,\quad \beta\ge1,
\end{equation}
where $h_e$ is the facet diameter. The penalty exponent $\beta$ controls the asymptotic mesh scaling, whereas the prefactor $\sigma$ controls the 
stabilization strength on a fixed mesh.
In fact, the parameterization in \eqref{eq:mu} separates two distinct roles of the penalty term. The prefactor $\sigma$ controls the baseline stabilization 
strength, while the exponent $\beta$ determines how the stabilization scales under mesh refinement. In particular, $\sigma$ governs the strength of 
interelement jump penalization on a fixed mesh, whereas $\beta$ dictates how rapidly the associated stiffness increases as $h_e \to 0$, with direct 
implications for the conditioning of the discrete system.
This decomposition is useful both analytically and computationally. At the analytical level, it clarifies how the jump contribution enters continuity and coercivity estimates. At the computational level, it distinguishes between stabilization at a fixed resolution and mesh-dependent growth in the asymptotic regime. The comparison between $\beta=1$ and $\beta=3$ therefore represents a comparison between standard and super-penalized NIPG scaling laws rather than a simple parameter choice.

For the SIPG reference used in the manufactured-solution comparison, the third term in \eqref{eq7} is replaced by
\[
-\sum_{e\in\mathcal E_h}\int_e\{\!\{\nabla\mathbf w\}\!\}\mathbf n\cdot[[S(\mathbf z)]]\,\mathrm ds,
\]
and the penalty parameter $\mu_e$ is chosen as in \cite{ref_article4} while the remaining terms are unchanged.

\subsection{Semi-discrete algebraic form}

To facilitate reproducibility and computational implementation, we express the semi-discrete formulation in algebraic matrix form. The approximate solution $\mathbf{q}_h$ is expanded using local basis functions. Let $\{\varphi_i^K\}_{i=1}^{N_p}$ be a set of scalar basis functions spanning $V_k$, where $N_p=\mathrm{dim}(V_k)$. Each component of $\mathbf{q}_h$ is locally represented as
\begin{equation}
    \left(\mathbf{q}_h\right)_\alpha|_K (\mathbf{x},t)
=
\sum_{i=1}^{N_p} c_{\alpha,i}^K(t)\,\varphi_i^K(\mathbf{x}),
\end{equation}
where $\alpha=1,2,3$ denotes the three components corresponding to $\phi_\eta$, $U$, and $V$, respectively.

For implementation, local coefficient vectors are first defined as
\[
\mathbf{c}_\alpha^K
=
(c_{\alpha,1}^K,\ldots,c_{\alpha,N_p}^K)^T,
\]
and then assembled into global vectors
\[
\mathbf{c}_\alpha
=
((\mathbf{c}_\alpha^{K_1})^T,(\mathbf{c}_\alpha^{K_2})^T,\ldots)^T,
\quad \alpha=1,2,3.
\]
The global degree of freedom (DOF) vector is ordered consistently with the physical variables and written as
\[
\mathbf{q}_{\mathrm{dof}}
=
[\phi_{\eta,\mathrm{dof}},\,U_{\mathrm{dof}},\,V_{\mathrm{dof}}]^T.
\]

Using a Galerkin projection with identical trial and test spaces, the semi-discrete system becomes
\begin{equation}
\frac{\mathrm{d}\mathbf{q}_{\mathrm{dof}}}{\mathrm{d}t}
=
\mathbf{M}^{-1}
\left(
\mathbf{P}_{\mathrm{vol}}
+
\mathbf{P}_{\mathrm{surf}}
+
\mathbf{P}_{\mathrm{src}}
\right),
\end{equation}
where the right-hand side corresponds to the volume, surface, and source contributions derived from $-b_h$, $-a_h$, and $L_h$, respectively.

The mass matrix has a block-diagonal structure,
\begin{equation}\notag
\mathbf{M}
=
\mathrm{diag}(\mathbf{M}_{\phi_\eta},\mathbf{M}_U,\mathbf{M}_V),
\end{equation}
with identical element-wise contributions due to the use of the same scalar basis for all variables,
\begin{equation}\notag
\mathbf{M}_{\phi_\eta},\mathbf{M}_U,\mathbf{M}_V
=
\bigoplus_{K\in Q_h}\mathbf{M}_K.
\end{equation}
The element mass matrix is computed on the reference element $\hat{K}$ using quadrature,
\begin{equation}\notag
\left(\mathbf{M}_K\right)_{ij}
=
|\det(\mathbf{J}_K)|
\sum_{l=1}^{N_q} w_l\,
\hat{\varphi}_i(\hat{\mathbf{x}}_l)\hat{\varphi}_j(\hat{\mathbf{x}}_l).
\end{equation}
This block structure enables local inversion of $\mathbf{M}_K$, leading to efficient and parallelizable time integration.

To illustrate the assembly, consider the volume contribution from $-b_h$. After integration by parts, the elementwise contribution for component $\alpha$ is
\begin{equation}\notag
(\mathbf{P}_{\mathrm{vol},\alpha}^K)_i
=
\int_{\hat{K}}
\mathcal{F}_\alpha(\mathbf{q}_h)\cdot
\left(\mathbf{J}_K^{-T}\hat{\nabla}\hat{\varphi}_i\right)
|\det(\mathbf{J}_K)|
\,\mathrm{d}\hat{\mathbf{x}},
\quad \alpha=1,2,3,
\end{equation}
where affine mappings from reference to physical elements are assumed. The surface and source terms follow analogously from \eqref{eq5} and \eqref{eq7}.

\section{Linearized penalty-dependent analytical framework}
\label{sec:analysis}

To analyze how the penalty design affects the behavior of the discretization, we consider a linearized viscous rotating shallow water model around a smooth reference state and isolate the diffusive contribution via an associated NIPG bilinear form $a_h$. This framework plays a key role in maintaining a hyperbolic-parabolic structure for stability and error analysis.

\subsection{Penalty-dependent DG seminorm and assumptions}

Recall that $V_k$ denotes the broken finite-element space of piecewise polynomials of degree at most $k$ on $Q_h$. 
For a discrete state $\mathbf{q}_h=(\phi_{\eta,h}, U_h, V_h)^T \in \mathbf V_h$, we isolate its momentum components into a momentum vector field
$\mathbf{U}_h = (U_h, V_h)^T$ and define the DG seminorm as follows:
\begin{equation}
\label{eq11}
\|\mathbf{q}_h\|_{\mathrm{DG}}^2
:=
\sum_{K\in Q_h}\|\nabla \mathbf{U}_h\|_{L^2(K)}^2
+
\sum_{e\in\mathcal{E}_h}
\frac{\sigma}{h_e^\beta}
\|[[\mathbf{U}_h ]]\|_{L^2(e)}^2.
\end{equation}
This seminorm will be used to derive penalty-dependent coercivity and continuity estimates for the NIPG bilinear form. The factor $h_e^{-\beta}$ determines how interelement stabilization scales with mesh refinement and plays a central role in continuity, coercivity, and stability estimates.
\begin{remark}
This is a seminorm on the full three-component state because the geopotential component is not diffused. It is a norm on the momentum subspace after imposing the relevant boundary or mean-value constraint.
\end{remark}

\begin{assumption}[Regularity and mesh]
\label{asm:regularity}
The exact solution satisfies $\mathbf{q}\in \left(H^{k+1}(\Omega)\right)^3$ for some $k\ge 1$. The mesh family $Q_h$ is shape-regular and quasi-uniform. The reference depth remains strictly positive,
\[
H_0:=\inf_{\Omega}(\eta_0+b)>0.
\]
The penalty prefactor satisfies $\sigma > 0$.
\end{assumption}

\begin{assumption}[Reference state for linearization]
\label{asm:reference}
The linearization is carried out around a smooth reference state $\mathbf q_{\mathrm{eq}}$ with strictly positive depth and coefficients bounded in $W^{1,\infty}(\Omega)$. 
In balance-oriented settings, this reference state may be chosen to satisfy
\[
f_c\mathbf{k}\times \mathbf{u}_{\mathrm{eq}} + g\nabla \eta_{\mathrm{eq}} = 0,
\]
where $\mathbf{u}_{\mathrm{eq}} = (U_{\mathrm{eq}}/\phi_{\mathrm{eq}}, V_{\mathrm{eq}}/\phi_{\mathrm{eq}})^\top$ and the equilibrium total geopotential is $\phi_{\mathrm{eq}} = \phi_{\eta,\mathrm{eq}} + \phi_b$.
\end{assumption}

To investigate the stability and diffusion properties of the numerical scheme, we analyze the system's behavior under small perturbations. We decompose $\mathbf{q}=\mathbf{q}_{\mathrm{eq}}+\mathbf{q}^\prime$, where $\mathbf{q}_{\mathrm{eq}}$ is a steady smooth reference state and $\mathbf{q}^\prime$ is a perturbation. Substituting into \eqref{eq3} yields the linearized system
\begin{equation}\label{eq12}
\partial_t \mathbf{q}^{\prime}+\partial_x\left(\mathbf{A}_1(\mathbf{q}_{\mathrm{eq}}) \mathbf{q}^{\prime}\right)+\partial_y\left(\mathbf{A}_2(\mathbf{q}_{\mathrm{eq}}) \mathbf{q}^{\prime}\right)-\Delta(\mathbf{B} \mathbf{q}^{\prime}) = \mathbf{C} \mathbf{q}^{\prime}+\mathbf{r}_{\mathrm{eq}} ,
\end{equation}
where
\[
\mathbf A_1=
\begin{pmatrix}
0&1&0\\
\phi_{\mathrm{eq}}-U_{\mathrm{eq}}^2/\phi_{\mathrm{eq}}^2&2U_{\mathrm{eq}}/\phi_{\mathrm{eq}}&0\\
-U_{\mathrm{eq}}V_{\mathrm{eq}}/\phi_{\mathrm{eq}}^2&V_{\mathrm{eq}}/\phi_{\mathrm{eq}}&U_{\mathrm{eq}}/\phi_{\mathrm{eq}}
\end{pmatrix},
\]
\[
\mathbf A_2=
\begin{pmatrix}
0&0&1\\
-U_{\mathrm{eq}}V_{\mathrm{eq}}/\phi_{\mathrm{eq}}^2&V_{\mathrm{eq}}/\phi_{\mathrm{eq}}&U_{\mathrm{eq}}/\phi_{\mathrm{eq}}\\
\phi_{\mathrm{eq}}-V_{\mathrm{eq}}^2/\phi_{\mathrm{eq}}^2&0&2V_{\mathrm{eq}}/\phi_{\mathrm{eq}}
\end{pmatrix},
\quad
\mathbf B=\operatorname{diag}(0,\nu_T,\nu_T),
\]
\[
\mathbf C=
\begin{pmatrix}
0&0&0\\
\partial_x\phi_b&0&f_c\\
\partial_y\phi_b&-f_c&0
\end{pmatrix}.
\]
The reference residual $\mathbf r_{\mathrm{eq}}$ is the full steady residual
\begin{equation}
\label{eq:reference-residual}
\mathbf r_{\mathrm{eq}}
=R(\mathbf q_{\mathrm{eq}})-\nabla\cdot\mathcal F(\mathbf q_{\mathrm{eq}})
+\Delta S(\mathbf q_{\mathrm{eq}}),
\end{equation}
that is,
\[
(\mathbf r_{\mathrm{eq}})_1=-\partial_xU_{\mathrm{eq}}-\partial_yV_{\mathrm{eq}},
\]
\[
\begin{aligned}
(\mathbf r_{\mathrm{eq}})_2={}&f_cV_{\mathrm{eq}}+\phi_{\eta,\mathrm{eq}}\partial_x\phi_b
-\partial_x\left(\frac{U_{\mathrm{eq}}^2}{\phi_{\mathrm{eq}}}
+\frac12(\phi_{\mathrm{eq}}^2-\phi_b^2)\right)
-\partial_y\left(\frac{U_{\mathrm{eq}}V_{\mathrm{eq}}}{\phi_{\mathrm{eq}}}\right)
+\nu_T\Delta U_{\mathrm{eq}},\\
(\mathbf r_{\mathrm{eq}})_3={}&-f_cU_{\mathrm{eq}}+\phi_{\eta,\mathrm{eq}}\partial_y\phi_b
-\partial_x\left(\frac{U_{\mathrm{eq}}V_{\mathrm{eq}}}{\phi_{\mathrm{eq}}}\right)
-\partial_y\left(\frac{V_{\mathrm{eq}}^2}{\phi_{\mathrm{eq}}}
+\frac12(\phi_{\mathrm{eq}}^2-\phi_b^2)\right)
+\nu_T\Delta V_{\mathrm{eq}}.
\end{aligned}
\]
Thus $\mathbf r_{\mathrm{eq}}=0$ whenever the reference state $\mathbf q_{\mathrm{eq}}$ is an exact steady solution.

\subsection{Consistency, continuity, and coercivity}\label{sec:Consistency}

\begin{lemma}[Consistency]
\label{lma:consistency}
Under Assumption~\ref{asm:regularity}, the exact solution satisfies the NIPG identity associated with the diffusive operator $-\Delta S(\mathbf q)$.
\end{lemma}
\begin{proof}
Let $\mathbf{q} \in (H^{k+1}(\Omega))^3$ be the exact solution. Then $S(\mathbf{q}) \in (H^{1}(\Omega))^3$, and in particular $S(\mathbf{q})$ is single-valued across all interior facets.

Testing the strong form $-\Delta S(\mathbf{q})$ with $\mathbf{w}_h$ and applying integration by parts on each element yields
\[
-\sum_{K \in Q_h}\int_K \Delta S(\mathbf{q})\,\mathbf{w}_h\,\mathrm{d}\mathbf{x}
=
\sum_{K \in Q_h}\int_K \nabla S(\mathbf{q}) \cdot \nabla \mathbf{w}_h\,\mathrm{d}\mathbf{x}
-
\sum_{K \in Q_h}\int_{\partial K} (\nabla S(\mathbf{q})\cdot \mathbf{n})\,\mathbf{w}_h\,\mathrm{d}s.
\]

Summing over all elements, the interior boundary contributions cancel due to the continuity of $S(\mathbf{q})$ and $\nabla S(\mathbf{q})$, i.e.
\[
[[S(\mathbf{q})]] = 0, \quad \text{and} \quad [[\nabla S(\mathbf{q})]] = 0 \quad \text{on } \Gamma_h.
\]

On boundary faces, $S(\mathbf{q}) = S(\mathbf{q}_{\mathrm{D}})$, and the resulting boundary terms coincide exactly with those appearing in $L_h$ defined in \eqref{eq5}. Hence all boundary contributions are consistently incorporated in the discrete formulation.

Finally, since $S(\mathbf{q})$ is continuous across interfaces, the penalty term
\[
\mu_e \sum_{e\in\mathcal{E}_h} \int_e [[S(\mathbf{q})]] \cdot [[\mathbf{w}_h]]\,\mathrm{d}s
\]
vanishes identically on interior faces and does not introduce any inconsistency on boundary faces due to the matching boundary contributions in $L_h$.
This completes the proof.
\end{proof}

\begin{lemma}[Continuity and coercivity]
\label{lma:coercivity}
Under Assumption~\ref{asm:regularity}, and for $\beta\ge 1$, there exists a constant $C_b=C \nu_T \cdot \max \left\{1,\frac{C_k \sqrt{n_0}}{\sqrt{\sigma}}\right\}>0$ independent of $h$ such that
\begin{equation}
\label{eq13}
|a_h(\mathbf{z}_h,\mathbf{w}_h)|
\le C_b\|\mathbf{z}_h\|_{\mathrm{DG}}\|\mathbf{w}_h\|_{\mathrm{DG}},
\quad
\forall \mathbf{z}_h,\mathbf{w}_h\in\mathbf{V}_h.
\end{equation}
For any $\sigma>0$, the coercivity property is unconditionally satisfied
\begin{equation}
\label{eq14}
a_h(\mathbf{z}_h,\mathbf{z}_h)=\nu_T\|\mathbf{z}_h\|_{\mathrm{DG}}^2.
\end{equation}
\end{lemma}

\begin{proof}
We first prove continuity. By the definition of the NIPG bilinear form $a_h$ in \eqref{eq7}, the viscous contribution consists of a volume gradient term and two interface terms, together with the penalty stabilization term. For $\mathbf{z}_h=(\zeta_h,m_{1,h},m_{2,h}),\mathbf{w}_h=(\xi_{h},w_{1,h},w_{2,h}) \in \mathbf V_h$, we write
\begin{equation}\notag
\begin{aligned}
a_h(\mathbf{z}_h,\mathbf{w}_h) \triangleq
I_1+I_2+I_3+I_4,
\end{aligned}
\end{equation}
where
\begin{equation}\notag
\begin{aligned}
I_1 &= \sum_{K \in Q_h}\int_K \nabla S(\mathbf{z}_h)\cdot \nabla\mathbf{w}_h\,\mathrm{d} \mathbf{x},
\\
I_2 &= -\sum_{e \in \mathcal{E}_h}\int_{e}\{\nabla S(\mathbf{z}_h)\} \cdot \mathbf{n}[[\mathbf{w}_h]]\,\mathrm{d}s,\\
I_3 &=
 \sum_{e \in \mathcal{E}_h}\int_{e}\{\nabla\mathbf{w}_h\} \cdot \mathbf{n}[[S(\mathbf{z}_h)]]\,\mathrm{d} s \\
 I_4 &= \mu_e \sum_{e \in \mathcal{E}_h} \int_{e}[[S(\mathbf{z}_h)]] \cdot [[\mathbf{w}_h]]\,\mathrm{d}s.
\end{aligned}
\end{equation}

For the gradient term, Cauchy--Schwarz inequality directly gives
\begin{equation}\notag
|I_1|
\le
\left(
\sum_{K\in Q_h}\|\nu_T\nabla (m_{1,h},m_{2,h})\|_{L^2(K)}^2
\right)^{1/2}
\left(
\sum_{K\in Q_h}\|\nabla (w_{1,h},w_{2,h})\|_{L^2(K)}^2
\right)^{1/2}.
\end{equation}
Since $\sigma>0$, this contribution is strictly bounded by $\nu_T\|\mathbf{z}_h\|_{\mathrm{DG}}\|\mathbf{w}_h\|_{\mathrm{DG}}$.

For the interface terms, applying Cauchy--Schwarz inequality on each face yields
\begin{equation}\notag
|I_2|
\le
\left(
\sum_{e\in \mathcal{E}_h}\frac{\nu_T^2h_e^\beta}{\sigma}\|\{\nabla (m_{1,h},m_{2,h})\}\cdot \mathbf{n}\|_{L^2(e)}^2
\right)^{1/2}
\left(
\sum_{e\in \mathcal{E}_h}\frac{\sigma}{h_e^\beta}\|[[ (w_{1,h},w_{2,h})]]\|_{L^2(e)}^2
\right)^{1/2}.
\end{equation}
Here, due to the fact that $h_e \le h_K \triangleq \sup\limits_{p, q \in K}\|p-q\| \le h$, and assuming without loss of generality that $\beta \ge 1$ and $h_K \le 1$, the first factor is controlled through the discrete trace inequality and local mesh regularity:
\begin{equation}\notag
\begin{aligned}
&\sum_{e\in \mathcal{E}_h}\frac{\nu_T^2h_e^\beta}{\sigma}\|\{\nabla (m_{1,h},m_{2,h})\}\cdot \mathbf{n}\|_{L^2(e)}^2\\
\le& \sum_{e \in \mathcal{E}_h} \frac{\nu_T^2h_e^\beta}{\sigma}\sum_{K \in \{K^+,K^-\}}\frac{1}{2} \| \left.\nabla (m_{1,h},m_{2,h})\cdot \mathbf{n} \right|_{K}\|_{L^2(e)}^2  \\
\le& \frac{\nu_T^2C_k^2}{2\sigma}\sum_{e \in \mathcal{E}_h} h_e^\beta (h_{K^+}^{-1} + h_{K^-}^{-1}) \sum_{K \in \{K^+,K^-\}} \|\nabla (m_{1,h},m_{2,h})\|_{L^2(K)}^2 \\
\le& 
\frac{\nu_T^2C_k^2 n_0}{\sigma} \sum_{K\in Q_h} \|\nabla (m_{1,h},m_{2,h})\|_{L^2(K)}^2,
\end{aligned}
\end{equation}
where $n_0$ denotes the maximum number of neighbors an element can have, and the constant $C_k$ derived from the trace inequality is independent of $h_K$ and $\mathbf{z}_h$, but depends on the polynomial degree $k$. Then, the term $|I_2|$ is bounded by $\frac{\nu_T C_k \sqrt{n_0}}{\sqrt{\sigma}}\|\mathbf{z}_h\|_{\mathrm{DG}}\|\mathbf{w}_h\|_{\mathrm{DG}}$.

A similar argument based on the trace inequality yields the same bound for $I_3$. Finally, evaluating the penalty term $I_4$ with $\mu_e = \sigma/h_e^\beta$ satisfies
\[
|I_4|
\le \nu_T
\left(
\sum_{e\in \mathcal{E}_h}\frac{\sigma}{h_e^\beta}
\|[[(m_{1,h},m_{2,h})]]\|_{L^2(e)}^2
\right)^{1/2}
\left(
\sum_{e \in \mathcal{E}_h}\frac{\sigma}{h_e^\beta}
\|[[(w_{1,h},w_{2,h})]]\|_{L^2(e)}^2
\right)^{1/2}.
\]
Collecting these bounds proves the continuity estimate \eqref{eq13}, with the combined constant $C_b=C \nu_T \cdot \max \left\{1,\frac{C_k \sqrt{n_0}}{\sqrt{\sigma}}\right\}$.

Setting $\mathbf w_h=\mathbf z_h$ makes the two nonsymmetric consistency terms cancel exactly, leaving the volume and penalty parts in \eqref{eq14},
therefore leading to the coercivity.
\end{proof}

\subsection{Semidiscrete stability}

\begin{theorem}[Semidiscrete stability estimate]
\label{th:stability}
Under Assumptions~\ref{asm:regularity}, ~\ref{asm:reference} and using Lemma~\ref{lma:coercivity}, when $\beta \ge 1$, the linearized semidiscrete solution 
satisfies
\begin{equation}
\label{eq:stability}
\|\mathbf{q}_h^{\prime}(T)\|_{L^2(\Omega)}^2
+
\nu_T\int_0^T\|\mathbf{q}_h^{\prime}\|_{\mathrm{DG}}^2 \mathrm{d}t \le
c_1 \left(\|\mathbf{q}^{\prime}_h(0)\|_{L^2(\Omega)}^2
+\int_0^T\|r(\mathbf{q}_{\mathrm{eq}})\|_{L^2(\Omega)}^2 \mathrm{d}t\right),
\end{equation}
where the constant $c_1$ is independent of the mesh size $h$ but may depend on $T$, the coefficient bounds, the symmetrizer bounds, $f_c$, $\nu_T^{-1}$, and $\sigma$.
\end{theorem}

\begin{proof}
To derive this estimate, we test the semi-discrete system \eqref{eq12} with $\mathbf{w}_h=\mathbf{q}_h^\prime$. For the time-derivative term, we directly get
\[
\sum_{K\in Q_h}\int_K \mathbf{q}_h^{\prime} \cdot \partial_t \mathbf{q}_h^{\prime} \,\mathrm{d} \mathbf{x}
=
\frac{1}{2}\frac{\mathrm{d}}{\mathrm{d}t}\|\mathbf{q}^{\prime}_h(t)\|_{L^2(\Omega)}^2.
\]

The hyperbolic contribution is handled using the local Lax--Friedrichs flux. For the $x$-directional flux, substituting the numerical flux definition yields
\begin{equation}\notag
\begin{aligned}
&\sum_{K \in Q_h} \int_K (\partial_x\mathbf{q}_h^{\prime}) \cdot (\mathbf{A}_1\mathbf{q}_h^{\prime}) \, \mathrm{d} \mathbf{x} -\sum_{e \in \mathcal{E}_h} \int_e \widehat{\mathcal{F}}\left(\mathbf{q}_h^{\prime},n_x\right) \cdot [[\mathbf{q}_h^{\prime}]] \, \mathrm{d} s \\
=& \sum_{K \in Q_h} \int_K (\partial_x \mathbf{q}_h^{\prime}) \cdot (\mathbf{A}_{1} \mathbf{q}_h^{\prime})\, \mathrm{d} \mathbf{x} -\sum_{e \in \mathcal{E}_h} \int_e  \left( (\mathbf{A}_{1}\{\mathbf{q}_h^{\prime}\})n_x + \frac{\tau^*}{2} [[\mathbf{q}_h^{\prime}]]\right)\cdot [[\mathbf{q}_h^{\prime}]] \, \mathrm{d} s.
\end{aligned}
\end{equation}

Denote $\mathbf{A}_1=(a_{ij})$ and $\mathbf{A}_2=(\tilde{a}_{ij})$. The $x$-directional contributions involving $\mathbf{A}_1$ are estimated using the Cauchy--Schwarz, Young’s, and trace inequalities; the $y$-directional terms follow similarly.
\begin{equation*}
\begin{aligned}
&\sum_{K \in Q_h} \int_K  (\mathbf{A}_{1} \mathbf{q}_h^{\prime}) \cdot \partial_x \mathbf{q}_h^{\prime} \mathrm{d} \mathbf{x}  -\sum_{e \in \mathcal{E}_h} \int_e  \left( n_x \mathbf{A}_{1}\left\{\mathbf{q}_h^{\prime}\right\} + \frac{\tau^*}{2} [[\mathbf{q}_h^{\prime}]]\right) \cdot [[\mathbf{q}_h^{\prime}]]  \mathrm{d} s\\
=&\sum_{K \in Q_h} \int_K \left(\partial_x \phi_{\eta,h}^{\prime} a_{12} U_h^{\prime} +  \partial_x U_h^{\prime} \left( a_{21} \phi_{\eta,h}^{\prime} + a_{22} U_h^{\prime}\right)  + \partial_x V_h^{\prime} \left( a_{31} \phi_{\eta,h}^{\prime} + a_{32} U_h^{\prime} + a_{33} V_h^{\prime}\right) \right) \mathrm{d} \mathbf{x} \\
&- \sum_{e \in \mathcal{E}_h} \int_e \bigg( a_{12} \{U_h^{\prime}\}n_x \, [[\phi_{\eta,h}^{\prime}]] + \left(  a_{21}\{\phi_{\eta,h}^{\prime}\} + a_{22}\{U_h^{\prime}\} \right)n_x \, [[U_{h}^{\prime}]] \\
&+ \left(  a_{31}\{\phi_{\eta,h}^{\prime}\} + a_{32}\{U_h^{\prime}\} + a_{33}\{V_h^{\prime}\} \right) n_x \, [[V_{h}^{\prime}]] \bigg)\mathrm{d} s  \\
=&\sum_{e \in \mathcal{E}_h} \int_e a_{12} \{\phi_{\eta,h}^{\prime}\}n_x \, [[U_h^{\prime}]] \mathrm{d} \mathbf{s} + \sum_{K \in Q_h} \int_K  \bigg( - \partial_x a_{12} \phi_{\eta,h}^{\prime} \cdot U^{\prime}_h + \partial_x U_h^{\prime} \left((a_{21}-a_{12})\phi_{\eta,h}^{\prime} + a_{22}\cdot U_h^{\prime} \right) \\
&+ \partial_x V_h^{\prime} \left( a_{31} \phi_{\eta,h}^{\prime} + a_{32} U_h^{\prime} + a_{33} V_h^{\prime}\right) \bigg) \mathrm{d} \mathbf{x} - \sum_{e \in \mathcal{E}_h} \int_e \bigg( \left(  a_{21}\{\phi_{\eta,h}^{\prime}\} + a_{22}\{U_h^{\prime}\} \right)n_x \, [[U_{h}^{\prime}]]  \\
&+ \left(  a_{31}\{\phi_{\eta,h}^{\prime}\} + a_{32}\{U_h^{\prime}\} + a_{33}\{V_h^{\prime}\} \right)n_x \, [[V_{h}^{\prime}]] \bigg)\mathrm{d} s \\
\le&
\left( 3\frac{C_kC_a \epsilon}{\sigma}+2C_a \epsilon + C_a^{\prime} \epsilon\right) \|\phi_{\eta,h}^{\prime}\|_{L^2(\Omega)}^2
+ \left(  2C_a \epsilon + C_a^{\prime}C_{\epsilon} + \frac{C_k C_a\epsilon}{\sigma} \right) \|U_{h}^{\prime}\|_{L^2(\Omega)}^2  \\
&+\left( C_a \epsilon + \frac{C_k C_a\epsilon}{\sigma} \right) \|V_{h}^{\prime}\|_{L^2(\Omega)}^2 + 2C_a C_\epsilon \|\mathbf{q}_h^{\prime}\|_{\mathrm{DG}}^2 -
\frac{\tau^*}{2}\sum_{e \in \mathcal{E}_h}
\left\| [[\mathbf{q}_h^{\prime}]] \right\|^2_{L^2(e)},
\end{aligned}
\end{equation*}
where $C_{a}$ and $C_a^{\prime}$ denote the uniform bounds of $a_{ij}$ (or $\tilde{a}_{ij}$) and their spatial derivatives $\partial_x a_{ij}$ (or $\partial_y \tilde{a}_{ij})$ over $\Omega$, respectively. Summing the estimates in both spatial directions and choosing a suitable parameter $\epsilon>0$, the gradient terms can be absorbed, yielding
\begin{equation}\notag
\begin{aligned}
    \sum_{K \in Q_h} & \int_K \partial_x\mathbf{q}_h^{\prime} \cdot (\mathbf{A}_1\mathbf{q}_h^{\prime}) \mathrm{d} \mathbf{x} -\sum_{e \in \mathcal{E}_h} \int_e  [[\mathbf{q}_h^\prime]]\cdot \widehat{\mathcal{F}}\left(\mathbf{q}_h^{\prime},n_x\right) \mathrm{d} s \\
                     &\le C_{\mathrm{hy}} \|\mathbf{q}_{h}^{\prime}\|_{L^2(\Omega)}^2 + \frac{\nu_T}{2} \|\mathbf{q}_h^{\prime}\|_{\mathrm{DG}}^2  -
\frac{\tau^*}{2}\sum_{e \in \mathcal{E}_h}
\left\| [[\mathbf{q}_h^{\prime}]] \right\|^2_{L^2(e)}.
\end{aligned}
\end{equation}

The viscous contribution is controlled by Lemma~\ref{lma:coercivity}. In particular, the coercivity property implies
\[
a_h(\mathbf{q}_h^{\prime},\mathbf{q}_h^{\prime})= \nu_T\|\mathbf{q}_h^{\prime}\|_{\mathrm{DG}}^2,
\]
where only the momentum components $(U_h,V_h)$ enter the DG--type norm defined in \eqref{eq11}. This highlights that the penalty-dependent coercivity is fully embedded in this estimate.

Based on Assumption~\ref{asm:reference}, since both $\mathbf{q}$ and $\mathbf{q}_{\mathrm{eq}}$ satisfy identical physical boundary conditions on $\Gamma_{\mathrm{D}}$, the perturbation $\mathbf{q}^\prime$ satisfies homogeneous boundary conditions. The linear reaction term $\mathbf{C}\mathbf{q}_h^\prime$ in \eqref{eq12} is bounded using the uniform boundedness of $\mathbf{C}$ (whose entries $\partial_x\phi_b$, $\partial_y\phi_b$, and $f_c$ lie in $L^\infty(\Omega)$ by Assumption~\ref{asm:reference}),
\begin{equation}\notag
\sum_{K \in Q_h}\int_{K} \mathbf{q}_h^{\prime}\cdot(\mathbf{C}\mathbf{q}_h^\prime)\,\mathrm{d}\mathbf{x} \le C_f \|\mathbf{q}_h^\prime\|_{L^2(\Omega)}^2, \qquad C_f := \|\mathbf{C}\|_{L^\infty(\Omega)}.
\end{equation}
The residual forcing $\mathbf{r}(\mathbf{q}_{\mathrm{eq}})$ is controlled by Cauchy--Schwarz and Young's inequality,
\begin{equation}\notag
\sum_{K \in Q_h}\int_{K} \mathbf{q}_h^{\prime}  \cdot \mathbf{r}(\mathbf{q}_{\mathrm{eq}}) \,\mathrm{d} \mathbf{x} \leq \left\|\mathbf{q}_h^{\prime}\right\|_{L^2(\Omega)}\left\|\mathbf{r}(\mathbf{q}_{\mathrm{eq}})\right\|_{L^2(\Omega)} \leq \frac{1}{2} \left\|\mathbf{q}_h^{\prime}\right\|_{L^2(\Omega)}^2+\frac{1}{2}\left\|\mathbf{r}(\mathbf{q}_{\mathrm{eq}})\right\|_{L^2(\Omega)}^2.
\end{equation}

Combining the bounds above yields
\begin{equation}\notag
\begin{aligned} 
\frac{d}{dt}\|\mathbf{q}_h^{\prime}\|_{L^2(\Omega)}^2
+
\nu_T\|\mathbf{q}_h^{\prime}\|_{\mathrm{DG}}^2 &\le
2(C_{\mathrm{hy}}+C_f + \frac{1}{2})\|\mathbf{q}^{\prime}_h\|_{L^2(\Omega)}^2
+\|r(\mathbf{q}_{\mathrm{eq}})\|_{L^2(\Omega)}^2.
\end{aligned}
\end{equation}
Applying Grönwall’s inequality then completes the proof.
\end{proof}

\subsection{Penalty-dependent priori error estimate}
We now derive a priori error estimate for the NIPG formulation.
Let $P_h$ be the scalar $L^2$ projection for the geopotential and let $R_h$ be the NIPG elliptic projection for the momentum, defined by
\[
a_h\bigl((0,\mathbf m-R_h\mathbf m),(0,\mathbf v_h)\bigr)=0
\qquad\forall\mathbf v_h\in V_k^2,
\]
with compatible projected boundary data when $\Gamma_{\mathrm D}\ne\varnothing$. Set $\Pi_h\mathbf q'=(P_h\phi_\eta',R_h\mathbf m')$.

\begin{theorem}[A priori error estimate]
\label{th:error}
Under the assumptions of Theorem~\ref{th:stability}, if $\beta \ge 1$, the DG-NIPG approximation satisfies an estimate of the form
\begin{equation}
\label{error}
\begin{aligned}
\left(\int_0^T\|\mathbf{q}^\prime-\mathbf{q}_h^{\prime}\|_{\mathrm{DG}}^2 \mathrm{d} t\right)^{1/2}  &\le Ch^{k}\left(\left\|\mathbf{q}^\prime\right\|_{H^1(0,T;H^{k+1}(Q_h))} +
\left\|\mathbf{q}^\prime  \right\|_{L^2(0,T;H^{k+1}(Q_h))} \right),\\
\|(U^{\prime},V^\prime)-(U_h^{\prime},V_h^\prime)\|_{L^\infty(0,T;L^2(\Omega))}  &\le Ch^{(k+1)-\delta}\left(\left\|\mathbf{q}^\prime\right\|_{H^1(0,T;H^{k+1}(Q_h))} +
\left\|\mathbf{q}^\prime  \right\|_{L^2(0,T;H^{k+1}(Q_h))} \right),\\
\|\phi_\eta^\prime-\phi_{\eta,h}^\prime\|_{L^\infty(0,T;L^2(\Omega))}  &\le Ch^{k+1-\delta}\left\| \mathbf{q}^\prime\right\|_{H^1(0,T;H^{k+1}(Q_h))} +Ch^{k}
\left\|\mathbf{q}^\prime  \right\|_{L^2(0,T;H^{k+1}(Q_h))},
\end{aligned}
\end{equation}
where the constant $C$ is independent of the mesh size $h$ and $\delta=0$ for NIPG if $\beta \geq 3$, if the mesh consists only of triangles and tetrahedra, and if $g_{\mathrm{D}} \in V_k^2$. Otherwise, $\delta=1$.
\end{theorem}

\begin{proof}
Since $\boldsymbol{\Pi}_h \mathbf{q}^\prime$ denotes the mixed projection of the exact solution $\mathbf{q}^\prime$ onto $\mathbf{V}_h$, we 
decompose $\mathbf q'-\mathbf q_h'=(\mathbf q'-\Pi_h\mathbf q')-(\mathbf q_h'-\Pi_h\mathbf q')=\boldsymbol\chi-\boldsymbol\xi$,
with the scalar components defined as $\phi_\eta^{\prime}-\phi_{\eta,h}^{\prime}=:\boldsymbol{\chi}_\phi-\boldsymbol{\xi}_\phi$, $U^{\prime}-U_{h}^{\prime}=:\boldsymbol{\chi}_U-\boldsymbol{\xi}_U$, and $V^{\prime}-V_{h}^{\prime}=:\boldsymbol{\chi}_V-\boldsymbol{\xi}_V$.
For brevity, we denote the momentum error vectors as $\boldsymbol{\chi}_{U,V} = (\boldsymbol{\chi}_U, \boldsymbol{\chi}_V)^\top$ and $\boldsymbol{\xi}_{U,V} = (\boldsymbol{\xi}_U, \boldsymbol{\xi}_V)^\top$.
Correspondingly, the subscript $(\cdot)_{U,V}$ applied to any flux vector denotes its restriction to the second and third momentum components.

Subtracting the semidiscrete formulation for $\mathbf{q}^{\prime}_h$ from the exact variational identity satisfied by $\mathbf{q}^{\prime}$ yields the error equation:
\begin{equation}\notag
(\partial_t (\mathbf{q}^{\prime}-\mathbf{q}_h^{\prime}), \mathbf{w}_h) + b_h(\mathbf{q}^{\prime}-\mathbf{q}_h^{\prime}, \mathbf{w}_h) + a_h(\mathbf{q}^{\prime}-\mathbf{q}_h^{\prime}, \mathbf{w}_h) = L_h(\mathbf{q}^{\prime}-\mathbf{q}_h^{\prime}, \mathbf{w}_h).
\end{equation}
Substituting $\mathbf{q}^{\prime}-\mathbf{q}_h^{\prime} = \boldsymbol{\chi} - \boldsymbol{\xi}$ and choosing the test function $\mathbf{w}_h = \boldsymbol{\xi} \in \mathbf{V}_h$, we obtain
\begin{equation}\notag
(\partial_t \boldsymbol{\xi}, \boldsymbol{\xi}) + b_h(\boldsymbol{\xi}, \boldsymbol{\xi}) +a_h(\boldsymbol{\xi}, \boldsymbol{\xi}) = (\partial_t \boldsymbol{\chi}, \boldsymbol{\xi}) +  b_h(\boldsymbol{\chi}, \boldsymbol{\xi}) + a_h(\boldsymbol{\chi}, \boldsymbol{\xi}) - L_h(\boldsymbol{\chi} - \boldsymbol{\xi}, \boldsymbol{\xi}).
\end{equation}
Note that $L_h$ does not act on the boundary for the error equation since we assume the projection satisfies the boundary conditions, leading to $L_h(\boldsymbol{\chi}, \boldsymbol{\xi}) = (\mathbf{C}\boldsymbol{\chi}, \boldsymbol{\xi})$.

We begin by analyzing the error equation for the momentum variables. Following the stability analysis in Theorem~\ref{th:stability}, we obtain
\begin{equation}\notag
\begin{aligned}
    \frac{1}{2}&\frac{\mathrm{d}}{\mathrm{d}t}\|\boldsymbol{\xi}_{U,V}\|_{L^2(\Omega)}^2
+ \frac{\nu_T}{2}\|\boldsymbol{\xi}\|_{\mathrm{DG}}^2 -
(C_{\mathrm{hy}}+ C_f + \frac{1}{2})\|\boldsymbol{\xi}_{U,V}\|_{L^2(\Omega)}^2 \\
               &\le (\partial_t \boldsymbol{\chi}_{U,V}, \boldsymbol{\xi}_{U,V}) + b_h(\boldsymbol{\chi}_{U,V}, \boldsymbol{\xi}_{U,V}) + (\mathbf{C}\boldsymbol{\chi}_{U,V}, \boldsymbol{\xi}_{U,V}).
\end{aligned}
\end{equation}

We now bound the term $b_h(\boldsymbol{\chi}_{U,V}, \boldsymbol{\xi}_{U,V})$. After integrating the volume terms by parts and applying Cauchy--Schwarz's and Young’s inequalities on the resulting interface contributions, the jump products $[[\boldsymbol{\chi}]] \cdot [[\boldsymbol{\xi}]]$ are estimated by using a trace inequality for $\boldsymbol{\chi}$, while the DG--type norm is extracted for $\boldsymbol{\xi}$ (under the assumption $\beta\ge 1$). This yields, for the $x$-directional contribution,
\begin{equation}\notag
\begin{aligned}
&b_h^x(\boldsymbol{\chi}_{U,V}, \boldsymbol{\xi}_{U,V})\\
=&\sum_{K \in Q_h} \int_K \partial_x \boldsymbol{\xi}_{U,V} \cdot \left(\mathbf{A}_{1} \boldsymbol{\chi}\right)_{U,V} \mathrm{d} \mathbf{x}  -\sum_{e \in \mathcal{E}_h} \int_e  \left( \mathbf{n}_x \mathbf{A}_{1}\left\{\boldsymbol{\chi}\right\} + \frac{\tau^*}{2} [[\boldsymbol{\chi}]]\right)_{U,V} \cdot [[\boldsymbol{\xi}_{U,V}]]  \mathrm{d} s\\
=&\sum_{K \in Q_h} \int_K  \left( \partial_x \boldsymbol{\xi}_{U} \left( a_{21} \boldsymbol{\chi}_{\phi} + a_{22} \boldsymbol{\chi}_{U}\right)  + \partial_x \boldsymbol{\xi}_{V} \left( a_{31} \boldsymbol{\chi}_{\phi} + a_{32} \boldsymbol{\chi}_{U} + a_{33} \boldsymbol{\chi}_{V}\right)  \right) \mathrm{d} \mathbf{x}  \\
&- \sum_{e \in \mathcal{E}_h} \int_e \bigg( \left(  a_{21}\{\boldsymbol{\chi}_{\phi}\} + a_{22}\{\boldsymbol{\chi}_{U}\} \right)\mathbf{n}_x \, [[\boldsymbol{\xi}_{U}]] + \left(  a_{31}\{\boldsymbol{\chi}_{\phi}\} + a_{32}\{\boldsymbol{\chi}_{U}\} + a_{33}\{\boldsymbol{\chi}_{V}\} \right)\mathbf{n}_x \, [[\boldsymbol{\xi}_{V}]]  \\
&+ \frac{\tau^*}{2} [[\boldsymbol{\chi}_{U,V}]]\cdot [[\boldsymbol{\xi}_{U,V}]] \bigg) \mathrm{d} s\\
\le&
 \left(C_a \epsilon + \tau^*\epsilon \right)\| \boldsymbol{\xi}\|_{\mathrm{DG}}^2 + \left( 2C_a C_\epsilon + \frac{2C_k C_a\epsilon}{\sigma} \right) \|\boldsymbol{\chi}_{\phi}\|_{L^2(\Omega)}^2 + \frac{2C_k C_aC
 _\epsilon}{\sigma}h^2\|\partial_x \boldsymbol{\chi}_{\phi}\|_{L^2(\Omega)}^2  \\
&+ \left( 2C_a C_\epsilon + \frac{2C_k C_aC_\epsilon}{\sigma}+ \frac{2C_k C_\epsilon \tau^*}{\sigma}\right) \|\boldsymbol{\chi}_{U}\|_{L^2(\Omega)}^2+ \left(\frac{2C_k C_aC_\epsilon}{\sigma}+ \frac{2C_kC_\epsilon \tau^*}{\sigma} \right) h^2\|\partial_x \boldsymbol{\chi}_{U}\|_{L^2(\Omega)}^2 \\
& + \left( C_a C_\epsilon + \frac{C_k C_aC_\epsilon}{\sigma}+ \frac{2C_k C_\epsilon \tau^*}{\sigma}\right) \|\boldsymbol{\chi}_{V}\|_{L^2(\Omega)}^2 \left(\frac{C_k C_aC_\epsilon}{\sigma}+ \frac{2C_kC_\epsilon \tau^*}{\sigma} \right) h^2\|\partial_x \boldsymbol{\chi}_{V}\|_{L^2(\Omega)}^2.
\end{aligned}
\end{equation}
Therefore, the corresponding estimate for the $y$-direction follows analogously. By choosing a suitable parameter $\epsilon > 0$, we obtain
\begin{equation}\notag
\begin{aligned}
b_h(\boldsymbol{\chi}_{U,V}, \boldsymbol{\xi}_{U,V})
&\le \frac{\nu_T}{4} \|\boldsymbol{\xi}\|_{\mathrm{DG}}^2  + C_v h^{2}
\left\|\nabla \boldsymbol{\chi} \right\|^2_{L^2(\Omega)} + \bar{C}_v 
\left\|\boldsymbol{\chi} \right\|^2_{L^2(\Omega)},
\end{aligned}
\end{equation}
where the constants $C_v$ and $\bar{C}_v$ depend on the flux Jacobians, trace constants, and penalty parameters.

The source term is bounded by:
\begin{equation}\notag
\begin{aligned}
  \sum_{K \in Q_h} \int_K (\mathbf{C}\boldsymbol{\chi})_{U,V} \cdot \boldsymbol{\xi}_{U,V} \, \mathrm{d} \mathbf{x} \le \frac{C_L}{2}\left\|\boldsymbol{\chi}_{U,V}\right\|_{L^2(\Omega)}^2 + \frac{1}{2}\left\|\boldsymbol{\xi}_{U,V}\right\|_{L^2(\Omega)}^2,
\end{aligned}
\end{equation}
where $C_{L}$ depends on the matrix norm of $\mathbf{C}$.

Combining these bounds and applying Young's inequality to the time derivative term, we obtain
\begin{equation}\notag
\begin{aligned}
\frac{\mathrm{d}}{\mathrm{d}t}\|\boldsymbol{\xi}_{U,V}\|_{L^2(\Omega)}^2
+ \frac{\nu_T}{2}\|\boldsymbol{\xi}\|_{\mathrm{DG}}^2 \le& 
2(C_{\mathrm{hy}}+ C_f + 1)\|\boldsymbol{\xi}_{U,V}\|_{L^2(\Omega)}^2 + \left\|\partial_t \boldsymbol{\chi}_{U,V}\right\|^2_{L^2(\Omega)} \\
&+ 2C_v h^{2}
\left\|\nabla \boldsymbol{\chi} \right\|^2_{L^2(\Omega)} + \left(2\bar{C}_v  + C_L\right)
\left\|\boldsymbol{\chi} \right\|^2_{L^2(\Omega)}.
\end{aligned}
\end{equation}

From \cite{ref_article2}, if the exact solution $\mathbf{q}^{\prime}$ is sufficiently regular, the interpolation errors satisfy
\begin{equation}\notag
\begin{aligned}
&\forall t \geq 0, \quad\|\boldsymbol{\chi}\|_{\mathrm{DG}} \leq C h^{k}\|\mathbf{q}^{\prime}_{U,V}(t)\|_{H^{k+1}\left(Q_h\right)}\leq C h^{k}\|\mathbf{q}^{\prime}(t)\|_{H^{k+1}\left(Q_h\right)}, \\
&\forall t \geq 0, \quad\|\boldsymbol{\chi}_{\phi}\|_{L^2(\Omega)} \leq C h^{k+1}\|\phi_\eta^{\prime}(t)\|_{H^{k+1}\left(Q_h\right)}, \\
&\forall t \geq 0, \quad\|\boldsymbol{\chi}_{U,V}\|_{L^2(\Omega)} \leq C h^{k+1-\delta}\|\left(U^\prime,V^\prime \right)(t)\|_{H^{k+1}\left(Q_h\right)},
\end{aligned}
\end{equation}
where $\delta=0$ if $\beta \ge 3$ and specific mesh conditions are met, otherwise $\delta=1$.

Integrating the error equation from $0$ to $T$, assuming $\boldsymbol{\xi}(0)=0$, and substituting the interpolation bounds yields
\begin{equation}\notag
\begin{aligned}
&\|\boldsymbol{\xi}_{U,V}(T)\|_{L^2(\Omega)}^2
+ \frac{\nu_T}{2}\int_0^T\|\boldsymbol{\xi}\|_{\mathrm{DG}}^2 \, \mathrm{d} t \\
\le& C \int_0^T\left\|\boldsymbol{\xi}_{U,V}\right\|^2_{L^2(\Omega)} \, \mathrm{d} t + Ch^{2(k+1)-2\delta}\left(\left\|\partial_t \mathbf{q}^\prime\right\|^2_{L^2(0,T;H^{k+1}(Q_h))} +
\left\|\mathbf{q}^\prime  \right\|^2_{L^2(0,T;H^{k+1}(Q_h))} \right).
\end{aligned}
\end{equation}
Applying Gr\"onwall's inequality provides the bound for $\boldsymbol{\xi}_{U,V}$. The triangle inequality then gives the final estimates for the momentum components in $L^\infty(0,T;L^2(\Omega))$ and the DG norm.

Finally, we analyze the $L^2$ error for $\phi_\eta^\prime$. Testing the continuity equation with $\boldsymbol{\xi}_\phi$ yields
\begin{equation}\notag
\frac{1}{2}\frac{\mathrm{d}}{\mathrm{d}t}\|\boldsymbol{\xi}_{\phi}\|_{L^2(\Omega)}^2
 - C\|\boldsymbol{\xi}_{\phi}\|_{L^2(\Omega)}^2 \le b_h(\boldsymbol{\chi}_{\phi}, \boldsymbol{\xi}_{\phi}),
\end{equation}
where the estimate is similar as
\begin{equation}\notag
\begin{aligned}
b_h(\boldsymbol{\chi}_{\phi}, \boldsymbol{\xi}_{\phi}) &\le
 \left( \frac{2C_k C_a\epsilon}{\sigma} + 2C_a \epsilon + \tau^*C_k\epsilon \right)\| \boldsymbol{\xi}_\phi\|_{L^2(\Omega)}^2 
+ 2\tau^* C_k C_\epsilon h_K^{-1}\| \boldsymbol{\chi}_{\phi}\|_{L^2(\Omega)}^2 \\
& \quad + 2\tau^* C_k C_\epsilon h_K \|\nabla \boldsymbol{\chi}_{\phi}\|_{L^2(\Omega)}^2  
+  C_a^\prime C_\epsilon \|\boldsymbol{\chi}_{U,V}\|_{L^2(\Omega)}^2 + C_a C_\epsilon \|\boldsymbol{\chi}\|_{\mathrm{DG}}^2.
\end{aligned}
\end{equation}
Note the specific $h_K^{-1}$ and $h_K$ scaling factors are derived from the trace theorem. 

Combining all bounds, applying Gr\"onwall's inequality for $\boldsymbol{\xi}_\phi$, and using the triangle inequality provides the final result
\begin{equation}\notag
\begin{aligned}
\|\phi_\eta^\prime-\phi_{\eta,h}^\prime\|_{L^\infty(0,T;L^2(\Omega))}  \le Ch^{k+1-\delta}\left\|\mathbf{q}^\prime\right\|_{H^1(0,T;H^{k+1}(Q_h))} + Ch^{k} \left\|\mathbf{q}^\prime \right\|_{L^2(0,T;H^{k+1}(Q_h))}.
\end{aligned}
\end{equation}
This completes the proof.
\end{proof}

\begin{remark}
For $\beta=1$, the mesh-independent theoretical statement is the adjoint-inconsistent $O(h^k)$ momentum $L^2$ bound. An optimal $O(h^{k+1})$ rate may 
occur for odd $k$ on particular structured meshes \cite{Houston2002,Larson2004}.
\end{remark}

\section{Time discretization}
\label{sec:time}

To advance the semi-discrete system in time, we employ the classical three-stage, third-order strong stability preserving Runge--Kutta (SSP-RK3) method. Let $\mathcal{L}_h(\mathbf{Q})$ denote the nonlinear spatial discretization operator, which incorporates the hyperbolic fluxes, the NIPG viscous terms, and the source contributions evaluated at the state $\mathbf{Q}$. The fully discrete update from $t_n$ to $t_{n+1} = t_n + \Delta t$ is given by
\begin{equation}
\label{eq:ssprk3}
\begin{aligned}
\mathbf{Q}^{(1)} &= \mathbf{Q}^{n} + \Delta t\,\mathcal{L}_h(\mathbf{Q}^{n}), \\
\mathbf{Q}^{(2)} &= \frac{3}{4}\mathbf{Q}^{n} + \frac{1}{4}\mathbf{Q}^{(1)} + \frac{1}{4}\Delta t\,\mathcal{L}_h(\mathbf{Q}^{(1)}), \\
\mathbf{Q}^{n+1} &= \frac{1}{3}\mathbf{Q}^{n} + \frac{2}{3}\mathbf{Q}^{(2)} + \frac{2}{3}\Delta t\,\mathcal{L}_h(\mathbf{Q}^{(2)}).
\end{aligned}
\end{equation}
This explicit integrator is chosen for its strong stability preserving (SSP) property, which can help control spurious oscillations near sharp gradients 
or discontinuities.

Since the scheme is fully explicit, the time step is restricted by a Courant–Friedrichs–Lewy (CFL) condition. Following \cite{Anderson1995}, the local admissible time step on each element $K \in \Omega_h$ is defined as
\begin{equation}
\label{eq:cfl}
\Delta t_K = \delta \cdot \left[ \frac{|u|}{\Delta x} + \frac{|v|}{\Delta y} + \sqrt{gH}\sqrt{\frac{1}{\Delta x^2} + \frac{1}{\Delta y^2}} + 2 \nu_{T} \left( \frac{1}{\Delta x^2} + \frac{1}{\Delta y^2} \right) \right]_K^{-1}, \quad \forall K \in \Omega_h.
\end{equation}
Here, $\Delta x$ and $\Delta y$ denote the characteristic element sizes in the two coordinate directions. To ensure stability for higher-order DG discretizations, the Courant number is chosen as $\delta = 0.2/(k+1)$.

\section{Numerical experiments}
\label{sec:numerics}

This section reports four numerical experiments: a manufactured-solution convergence study, a penalty-sensitivity study, 
a smooth rotating benchmark, and two topography-aware tests. They provide numerical evidence about accuracy, penalty sensitivity, and selected balance properties.

The implementation is carried out within the open-source Firedrake framework \cite{Rathgeber2017}. 
It uses the Unified Form Language (UFL) \cite{Alnas2014} in Python to express variational forms, 
which are then translated into optimized low-level C code through automated code generation. All numerical simulations presented in this section are performed on an Ubuntu 22.04.5 LTS machine with 16 CPU cores and 27 GB of RAM. 

\subsection{Common numerical setting}
\label{sec:common-setting}

Unless otherwise stated, the computational domain is the unit square
\[
\Omega=[0,1]^2,
\]
with periodic boundary conditions. The mesh consists of uniform triangular meshes obtained by subdividing an $N\times N$ Cartesian grid into two triangles per cell, with $N=8,16,32,64$. The characteristic mesh size is denoted by $h$, while $h_e$ denotes the diameter of each edge $e \in \mathcal{E}_h$.

We consider polynomial degrees $k \in \{1,2\}$. The NIPG penalty parameter is defined as $\mu_e = \sigma h_e^{-\beta}$, where we compare 
the standard scaling ($\beta=1$) with the super-penalized choice ($\beta=3$). The additional value $\beta=2$ is considered in the sensitivity study. Unless otherwise specified, the gravitational constant is $g=9.81$, while the Coriolis parameter $f_c$, turbulent viscosity $\nu_T$, and final time $T$ are specified for each test case.

The numerical errors are measured using the component-wise $L^2$ norm
\begin{equation}
\label{eq:l2error}
E_{L^2}(\mathbf{q})=\left(\sum_{K\in Q_h}\int_K |\mathbf{q}-\mathbf{q}_h|^2\,\mathrm{d} \mathbf{x}\right)^{1/2},
\end{equation}
and by the DG-type momentum error
\begin{equation}
\label{eq:dgerror}
E_{\mathrm{DG}}(U,V)^2
=
\sum_{K\in Q_h}\|\nabla(U-U_h,V-V_h)\|_{L^2(K)}^2
+
\sum_{e\in\Gamma_h}\frac{\sigma}{h_e^\beta}
\|[[(U-U_h,V-V_h)]]\|_{L^2(e)}^2.
\end{equation}
Because the manufactured fields are symmetric under interchange of $x$ and $y$, the reported $U$ and $V$ errors agree to the displayed precision; only $E_{L^2}(U)$ is tabulated.
The observed convergence rate is computed as
\begin{equation}
\label{eq:observed-order}
\mathrm{order}=\frac{\log(E_h/E_{h/2})}{\log 2}.
\end{equation}

\subsection{Manufactured-solution convergence study}
\label{sec:manufactured}

The first experiment measures observed spatial rates and examines how the penalty-dependent estimates in Theorem~\ref{th:error} appear in the coupled system. The exact solution is prescribed as
\begin{equation}
\label{eq:manufactured_eta}
\eta(x,y,t)=1+0.05\sin(2\pi x)\sin(2\pi y)\cos t,
\end{equation}
\begin{equation}
\label{eq:manufactured_u}
\begin{aligned}
u(x,y,t)&=0.1\cos(2\pi x)\sin(2\pi y)\cos t,\\
v(x,y,t)&=0.1\sin(2\pi x)\cos(2\pi y)\cos t.
\end{aligned}
\end{equation}

The corresponding forcing terms are obtained by substituting the above into the viscous rotating shallow water system in geopotential variables. The bottom topography is set to $b(x,y)=0$, and the Coriolis parameter is taken as $f_c=0$. The final time is $T=0.01$. To assess the robustness of the scheme across different flow regimes, we vary the viscosity parameter as $\nu_T \in \{10^{-3}, 10^{-2}, 1\}$. The penalty prefactor is set to $\sigma=10$ unless otherwise specified; its influence is further investigated in Section~\ref{sec:penalty-sensitivity}.

\begin{table}[htbp]
\centering
\caption{$L^2$ errors, DG-type momentum errors and observed orders for the manufactured-solution convergence study when $\nu_T=0.01$.}
\label{tab1}
\small
\begin{tabular}{cccccccccc}
\toprule
$k$ & $\beta$ & $\Delta t$ & $h$ & $E_{L^2}(\phi_\eta)$ & order & $E_{L^2}(U)$ & order & $E_{\mathrm{DG}}(U,V)$ & order\\
\midrule
$1$ & $1$ & $2\times 10^{-4}$ & $1/8$ & $1.099 \times 10^{-2}$ & -- & $2.757\times 10^{-2}$ & -- & $2.535\times 10^{-1}$ & --\\
 & & & $1/16$ & $3.000 \times 10^{-3}$ & $1.874$ & $6.594\times 10^{-3}$ & $2.064$ & $1.309 \times 10^{-1}$ & $0.953$\\
 & & & $1/32$ & $7.775 \times 10^{-4}$ & $1.948$ & $1.488\times 10^{-3}$ & $2.147$ & $6.297\times 10^{-2}$ & $1.056$\\
  & & & $1/64$ & $2.034 \times 10^{-4}$ & $1.935$ & $3.444\times 10^{-4}$ & $2.112$  & $2.987\times 10^{-2}$ & $1.076$ \\
$1$ & $3$ & $1\times 10^{-7}$ & $1/8$ & $1.170 \times 10^{-2}$ & -- & $2.897\times 10^{-2}$ & -- & $3.402\times 10^{-1}$ & -- \\
 & & & $1/16$ & $3.418 \times 10^{-3}$ & $1.775$ & $6.537\times 10^{-3}$ & $2.148$ & $1.318\times 10^{-1}$ & $1.369$ \\
 & & & $1/32$ & $9.317 \times 10^{-4}$ & $1.875$ & $1.593\times 10^{-3}$ & $2.037$ & $6.172\times 10^{-2}$ & $1.094$ \\
  & & & $1/64$ & $2.429 \times 10^{-4}$ & $1.939$ & $3.957\times 10^{-4}$ & $2.009$  & $3.038\times 10^{-2}$ & $1.023$ \\
$2$ & $1$ & $1\times 10^{-4}$ & $1/8$ & $1.432 \times 10^{-3}$ & -- & $3.126\times 10^{-3}$ & -- & $4.026\times 10^{-2}$ & -- \\
 & & & $1/16$ & $1.788 \times 10^{-4}$ & $3.001$ & $4.234\times 10^{-4}$ & $2.884$ & $1.025\times 10^{-2}$ & $1.974$ \\
 & & & $1/32$ & $2.261 \times 10^{-5}$ & $2.984$ & $4.713\times 10^{-5}$ & $3.167$ & $2.463\times 10^{-3}$ & $2.057$ \\
  & & & $1/64$ & $3.170 \times 10^{-6}$ & $2.834$ & $5.423\times 10^{-6}$ & $3.119$  & $5.966\times 10^{-4}$ & $2.046$ \\
$2$ & $3$ & $5\times 10^{-8}$ & $1/8$ & $2.225 \times 10^{-3}$ & -- & $3.421\times 10^{-3}$ & -- & $5.299\times 10^{-2}$ & -- \\
 & & & $1/16$ & $5.961 \times 10^{-4}$ & $1.900$ & $5.481\times 10^{-4}$ & $2.642$ & $1.123\times 10^{-2}$ & $2.239$\\
 & & & $1/32$ & $1.510 \times 10^{-4}$ & $1.981$ & $7.067\times 10^{-5}$ & $2.955$ & $2.509\times 10^{-3}$ & $2.161$\\
  & & & $1/64$ & $3.738 \times 10^{-5}$ & $2.014$ & $8.463\times 10^{-6}$ & $3.023$  & $6.011\times 10^{-4}$ & $2.062$\\
\bottomrule
\end{tabular}
\end{table}

\begin{table}[htbp]
\centering
\caption{Robustness study: DG-type momentum errors, $L^2$ errors and convergence orders for weak ($\nu_T=10^{-3}$) and strong ($\nu_T=1$) viscous effects with fixed parameters: $k=2$, $\beta=1$, $\sigma=10$.}
\label{tab3}
\begin{tabular}{cccccccccc}
\toprule
$\nu_T$ & $\Delta t$ & $h$ & $E_{\mathrm{DG}}(U,V)$ & order & $E_{L^2}(\phi_\eta)$ & order & $E_{L^2}(U)$ & order \\
\midrule
$10^{-3}$ &  $2\times 10^{-4}$ & $1/8$ & $4.286\times 10^{-2}$ & -- & $1.412\times 10^{-3}$ & -- & $3.454\times 10^{-3}$ & -- \\
 & & $1/16$ & $1.161\times 10^{-2}$ & $1.885$ & $1.752\times 10^{-4}$ & $3.011$ & $5.348 \times 10^{-4}$ & $2.691$ \\
 & & $1/32$ & $2.862\times 10^{-3}$ & $2.020$ & $2.157\times 10^{-5}$ & $3.022$ & $6.493 \times 10^{-5}$ & $3.042$ \\
 & & $1/64$ & $6.879\times 10^{-4}$ & $2.057$ & $2.676\times 10^{-6}$ & $3.011$ & $7.664\times 10^{-6}$ & $3.083$ \\
$1$ &  $5\times 10^{-7}$ & $1/8$ & $3.395\times 10^{-2}$ & -- & $1.820\times 10^{-3}$ & -- & $5.120\times 10^{-3}$ & -- \\
 & & $1/16$ & $8.821\times 10^{-3}$ & $1.944$ & $2.951\times 10^{-4}$ & $2.625$ & $1.175 \times 10^{-3}$ & $2.124$ \\
 & & $1/32$ & $2.230\times 10^{-3}$ & $1.984$ & $5.023\times 10^{-5}$ & $2.555$ & $2.866 \times 10^{-4}$ & $2.035$ \\
 & & $1/64$ & $5.591\times 10^{-4}$ & $1.996$ & $9.948\times 10^{-6}$ & $2.336$ & $7.121\times 10^{-5}$ & $2.009$ \\
\bottomrule
\end{tabular}
\end{table}

\begin{figure}[htbp]
    \centering
    \begin{subfigure}{0.32\textwidth}
        \centering
        \includegraphics[width=\linewidth]{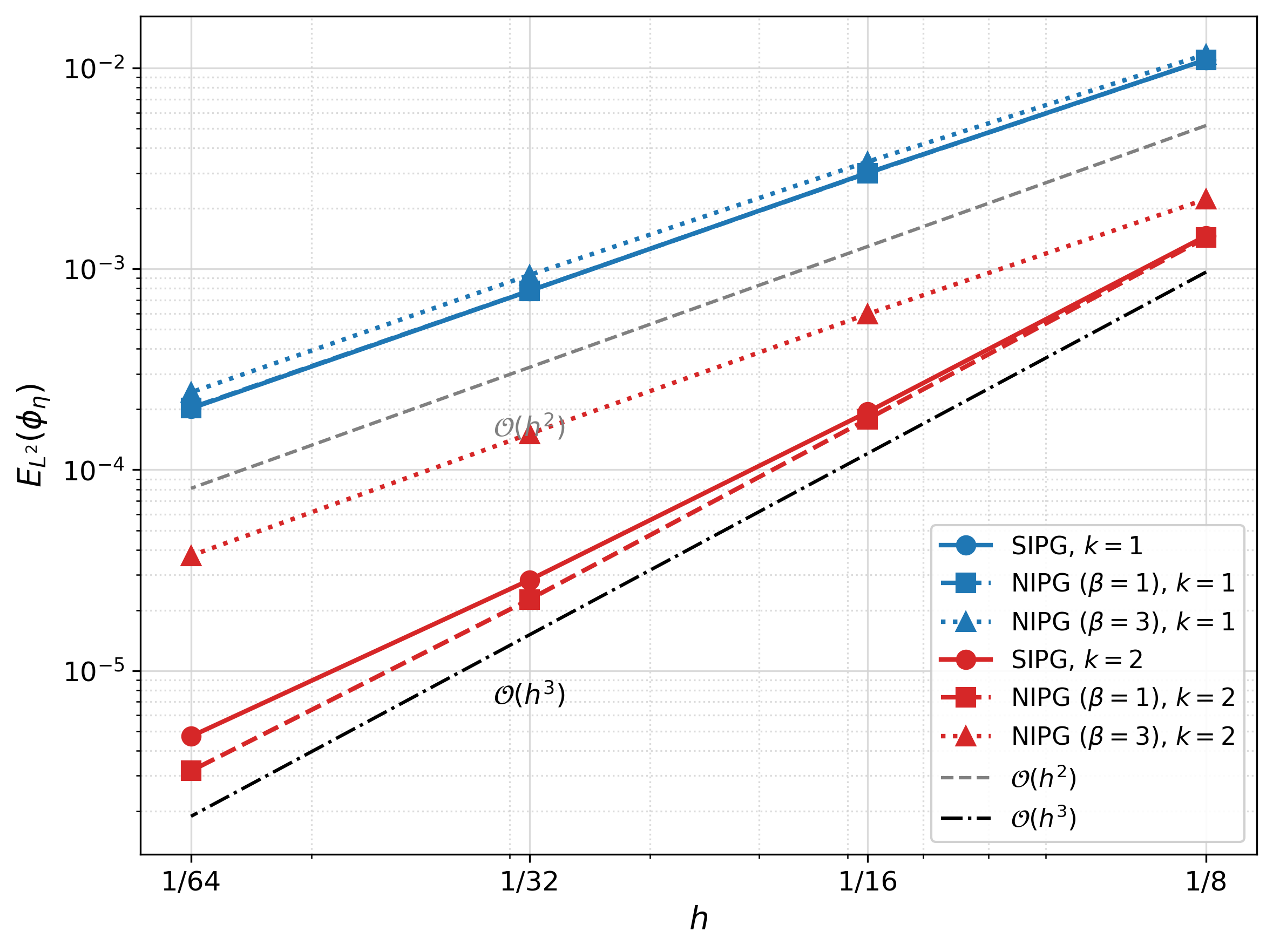}
    \end{subfigure}
    \hfill
    \begin{subfigure}{0.32\textwidth}
        \centering
        \includegraphics[width=\linewidth]{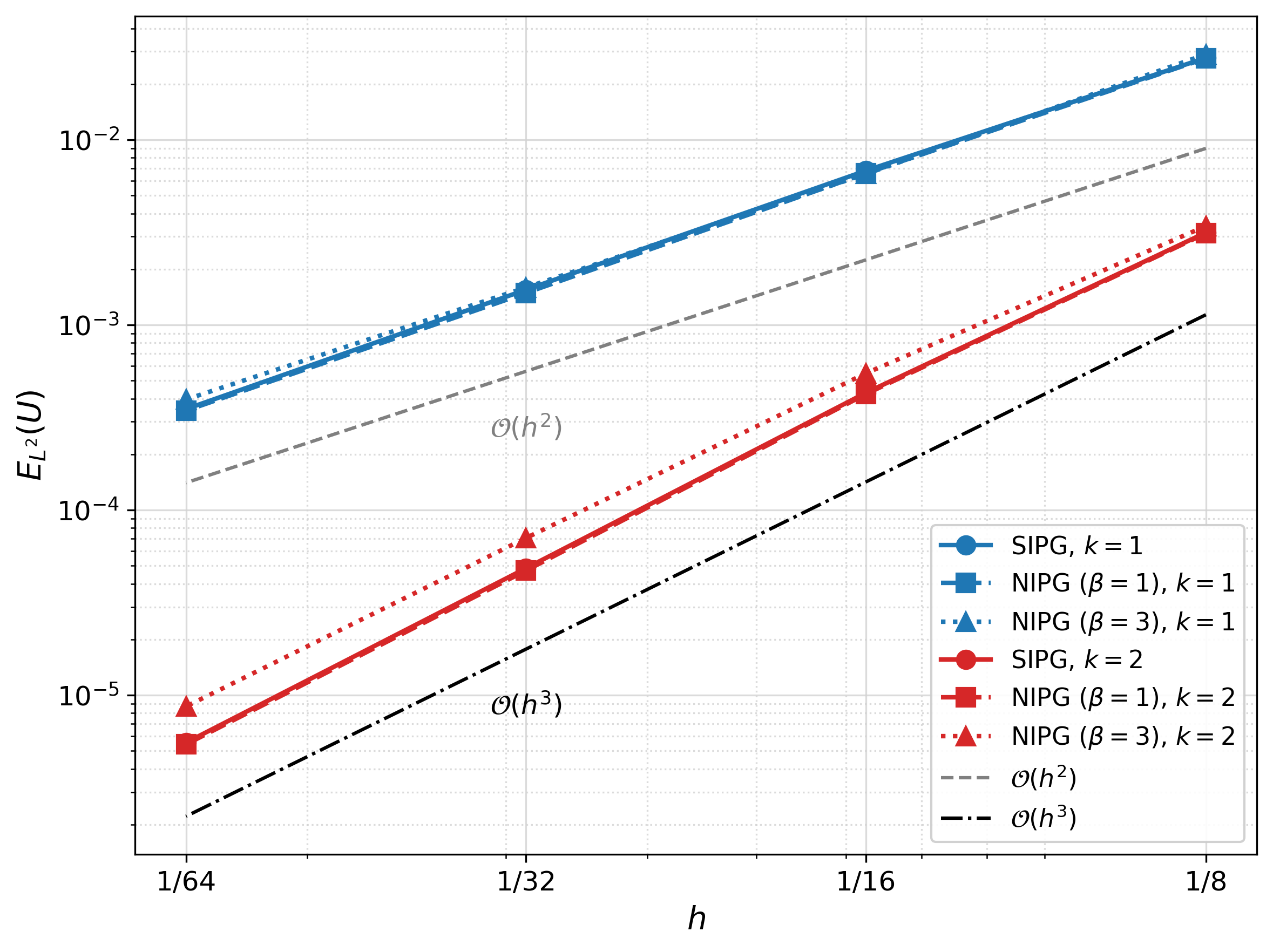}
    \end{subfigure}
    \hfill
    \begin{subfigure}{0.32\textwidth}
        \centering
        \includegraphics[width=\linewidth]{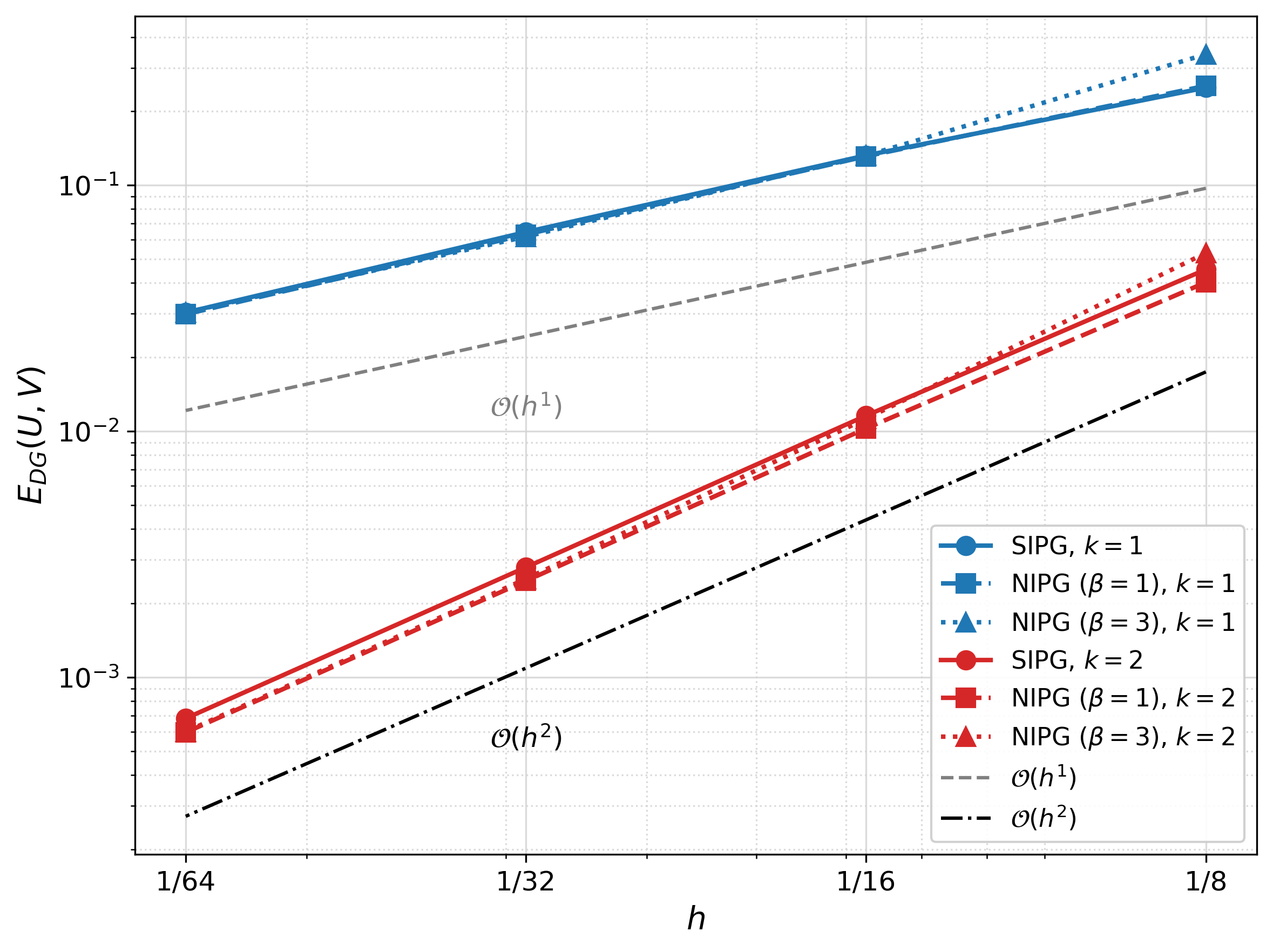}
    \end{subfigure}

    \caption{Convergence for the manufactured-solution study. Each plot includes curves for NIPG ($\beta=1$), NIPG ($\beta=3$), and SIPG for comparison across polynomial degrees $k=1,2$.}
    \label{fig1}
\end{figure}

Table~\ref{tab1} reports the spatial convergence rates of the $L^2$ errors and DG-type momentum errors, respectively, with $\nu_{T}=10^{-2}$. 
In this regime, the dynamics are primarily governed by advection and the diffusive NIPG operator contributes weakly to the overall error. 
As a result, the momentum variables achieve the optimal $O(h^{k+1})$ convergence rate in the $L^2$ norm for both penalty scalings. 
The geopotential error attains the optimal $O(h^{k+1})$ rate for $\beta=1$; for the super-penalized case $\beta=3$ it remains optimal at the odd degree $k=1$ but degrades to $O(h^{k})$ at the even degree $k=2$, in agreement with the odd/even behavior of NIPG discussed in the remark following Theorem~\ref{th:error}. In addition, the DG norm exhibits the expected $O(h^{k})$ behavior in all cases.

Table~\ref{tab3} presents the results with $\nu_{T}=1$ and $0.001$ for $k=2$ and $\beta=1$. In this case, when $\nu_{T}=1$, the suboptimal convergence rate is observed in the momentum $L^2$ error for even polynomial degree, which is consistent with the theoretical estimates for the NIPG formulation in \cite{ref_article2}.

Figure~\ref{fig1} compares the spatial convergence behavior of NIPG ($\beta=1,3$) with SIPG given in \cite{ref_article4} under polynomial degrees $k=1,2$. In all cases, the error curves are parallel to the reference slopes, confirming the theoretical convergence rates. Among these three methods, NIPG with $\beta=1$ generally achieves the smallest absolute errors, while $\beta=3$ produces the largest.

\begin{remark}
In geophysical fluid dynamics, physical regimes are typically characterized by high Reynolds numbers and weak viscosity \cite{Pedlosky2013}. 
To reflect this setting, we adopt a convection-dominated setting with $\nu_T=10^{-2}$, which preserves optimal convergence rates while avoiding viscosity-induced degradation. Moreover, the observed robustness of the scheme under this setting justifies its use for all subsequent numerical experiments.
The remaining experiments use $\nu_T=10^{-2}$ to focus on penalty-induced jump control and explicit stiffness.
\end{remark}

\subsection{Penalty-sensitivity study}
\label{sec:penalty-sensitivity}

This section presents a systematic penalty-sensitivity study to investigate how the prefactor $\sigma$ and the exponent $\beta$ influence the accuracy, stability, and numerical stiffness of the NIPG discretization. The goal is to verify that the observed numerical behavior is governed by a consistent mathematical mechanism rather than being an artifact of a particular tuning of parameters.

Using the manufactured solution from Section~\ref{sec:manufactured} with an $N \times N = 16 \times 16$ mesh, polynomial degree $k=2$, and final time $T=0.1$, the penalty exponent and prefactor are varied over $\beta \in \{1,2,3\}$ and $\sigma \in \{0.5, 1, 10, 100\}$, respectively. For each parameter pair $(\sigma,\beta)$, we monitor the geopotential error $E_{L^2}(\phi_\eta)$, the momentum errors $E_{L^2}(U)$ and $E_{L^2}(V)$, the DG-type norm $E_{\mathrm{DG}}(U,V)$ defined in \eqref{eq:l2error}–\eqref{eq:dgerror}, the maximum interface jump measure $J_{\max}$ defined by
\begin{equation}
\label{eq:jump_max}
J_{\max} = \max_{0\le t\le T} \left( \sum_{e\in\Gamma_h} \|[[ (U_h,V_h) ]]\|_{L^2(e)}^2 \right)^{1/2}.
\end{equation}
We also report the stable time step $\Delta t$ used by the explicit computation. The corresponding diagnostics are summarized in Table~\ref{tab4} and Figure~\ref{fig2}.

\begin{table}[htbp]
\centering
\caption{Penalty-sensitivity diagnostics for the NIPG method with $k=2$, $N=16$, $T=0.1$, and $\nu_T=10^{-2}$.}
\label{tab4}
\small
\begin{tabular}{ccccccccc}
\toprule
$k$ & $\beta$ & $\sigma$ &  $\Delta t$&$E_{L^2}(\phi_\eta)$ & $E_{L^2}(U)$ & $E_{\mathrm{DG}}(U,V)$ & $J_{\max}$ \\
\midrule
$2$ & $1.0$ & $0.5$ &  $1 \times 10^{-3}$ &$1.695\times 10^{-4}$ & $4.337 \times 10^{-4}$ & $2.710\times 10^{-2}$ & $7.295\times 10^{-3}$ \\
 & & $1.0$ &  $1 \times 10^{-3}$ & $1.699\times 10^{-4}$ & $4.326 \times 10^{-4}$ & $2.755 \times 10^{-2}$ & $7.220\times 10^{-3}$  \\
 & & $10.0$ &  $1 \times 10^{-3}$ &$1.789\times 10^{-4}$ & $4.234 \times 10^{-4}$ & $3.349\times 10^{-2}$ & $6.714\times 10^{-3}$ \\
 &  & $100.0$ &  $2 \times 10^{-4}$ &$4.936\times 10^{-4}$ & $4.580 \times 10^{-4}$ & $5.524\times 10^{-2}$ & $6.714\times 10^{-3}$ \\
 $2$ & $2.0$ & $0.5$ &  $1 \times 10^{-3}$ &$1.738\times 10^{-4}$ & $4.260 \times 10^{-4}$ & $3.099\times 10^{-2}$ & $6.714\times 10^{-3}$ \\
 & & $1.0$ &  $1 \times 10^{-3}$ &$1.809\times 10^{-4}$ & $4.231 \times 10^{-4}$ & $3.416\times 10^{-2}$ & $6.714\times 10^{-3}$ \\
 &  & $10.0$ &  $2 \times 10^{-4}$ &$5.387\times 10^{-4}$ & $4.635 \times 10^{-4}$ & $5.673\times 10^{-2}$ & $6.714\times 10^{-3}$ \\
 &  & $100.0$ &  $2 \times 10^{-5}$ &$1.855\times 10^{-3}$ & $6.638 \times 10^{-4}$ & $6.364\times 10^{-2}$ & $6.714\times 10^{-3}$ \\
 $2$ & $3.0$ & $0.5$ &  $2 \times 10^{-4}$ &$3.625\times 10^{-4}$ & $4.404 \times 10^{-4}$ & $4.978\times 10^{-2}$ & $6.714\times 10^{-3}$\\
 &  & $1.0$ &  $1 \times 10^{-4}$ & $5.878\times 10^{-4}$ & $4.693 \times 10^{-4}$ & $5.818\times 10^{-2}$ & $6.714 \times 10^{-3}$ \\
 &  & $10.0$ &  $2 \times 10^{-5}$ & $1.923\times 10^{-3}$ & $6.836 \times 10^{-4}$ & $6.278\times 10^{-2}$ & $6.714\times 10^{-3}$ \\
 & & $100.0$ &  $1 \times 10^{-6}$ &$2.609\times 10^{-3}$ & $8.892 \times 10^{-4}$ & $4.839\times 10^{-2}$ & $6.714\times 10^{-3}$ \\
\bottomrule
\end{tabular}
\end{table}

 As shown in Table~\ref{tab4}, large $\sigma$ or the super-penalized scaling $\beta = 3$ induces a pronounced deterioration in computational efficiency. The admissible time step $\Delta t$ decreases dramatically from $\mathcal{O}(10^{-4})$ to $\mathcal{O}(10^{-6})$ when $\sigma=100$, reflecting a severe stiffness induced by over-penalization in the explicit time integration.

In terms of spatial accuracy, the momentum error remains comparatively insensitive to the penalty parameters. This robustness is attributed to the strong suppression of interface jumps, as evidenced by the reduction of $J_{\max}$ in Table~\ref{tab4}. Once the inter-element discontinuities are sufficiently damped, the momentum error is primarily governed by the polynomial 
approximation capacity within each element.

By contrast, the geopotential error $E_{L^2}(\phi_\eta)$ is significantly affected by the super-penalty. This
behavior can be explained by the fact that the geopotential equation is not directly controlled by a diffusive NIPG stabilization, so that the coupled system amplifies the effect of
penalization through the coupling terms, resulting in a loss of balance between stabilization and approximation accuracy.
Figure~\ref{fig2} summarizes the same trend: excessive penalization does not improve the error at the fixed resolution considered here.

\begin{figure}[htbp]
\centering
\includegraphics[width=0.6\linewidth]{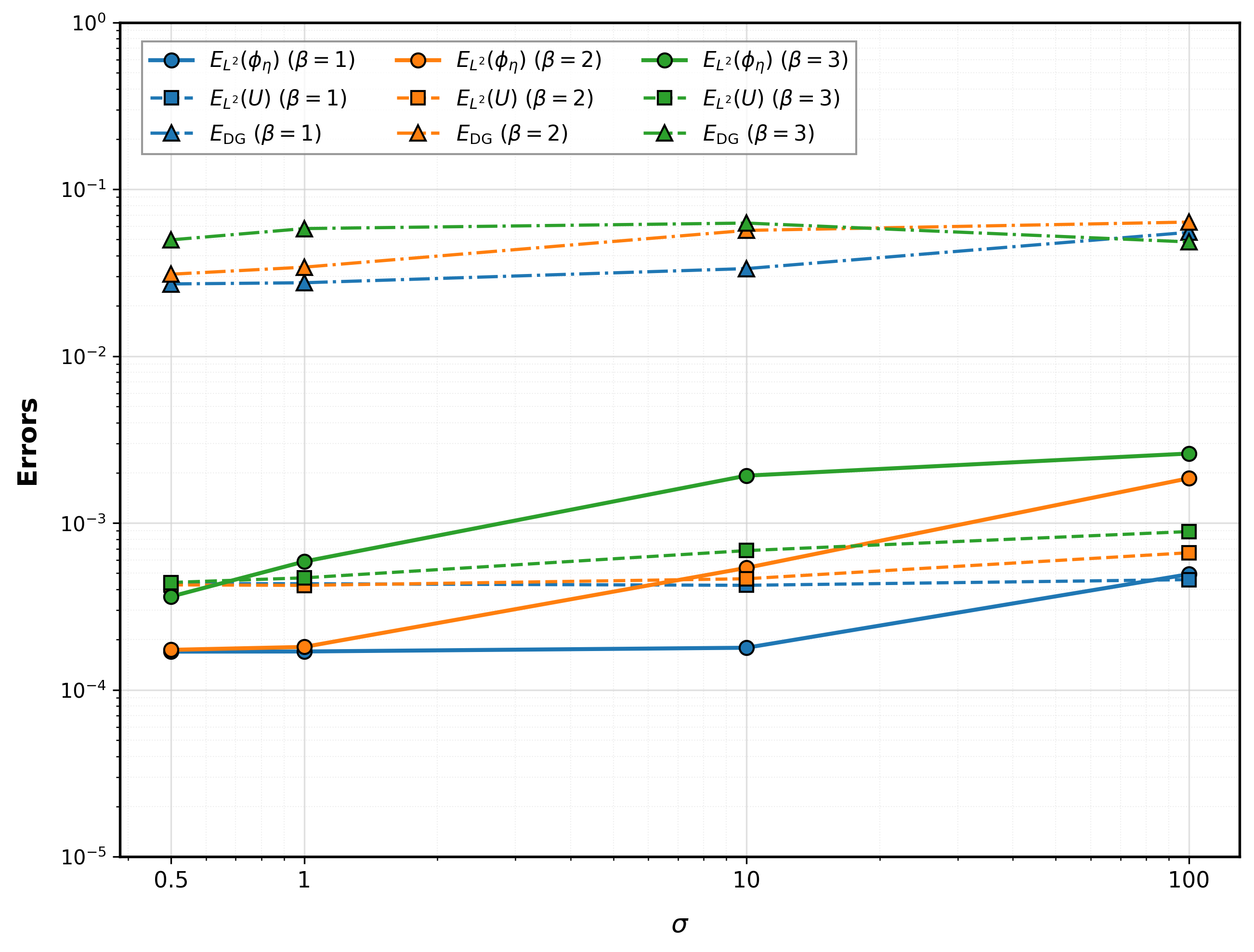}
\caption{Error versus penalty prefactor for selected penalty exponents.}
\label{fig2}
\end{figure}

\subsection{Smooth rotating benchmark with nonzero Coriolis forcing}
\label{sec:rotating-benchmark}

To evaluate the performance of the numerical scheme under rotational effects, the third experiment activates the Coriolis coupling ($f_c \ne 0$). This benchmark is based on a smooth geostrophically balanced reference state, commonly used for validating geophysical dynamical cores \cite{ambatibokhove2007,ref_article6,maddison2011}. We define the background geostrophic free surface as
\begin{equation}
\label{eq:eta_geo}
\eta_{\mathrm{geo}}(x,y)=1+a\cos(2\pi x)\cos(2\pi y),
\end{equation}
where the velocity fields satisfy the leading-order geostrophic balance relations
\begin{equation}
\label{eq:geostrophic_components}
-f_c v_{\mathrm{geo}}+g\frac{\partial \eta_{\mathrm{geo}}}{\partial x}=0,\qquad
f_c u_{\mathrm{geo}}+g\frac{\partial \eta_{\mathrm{geo}}}{\partial y}=0.
\end{equation}
Differentiating \eqref{eq:eta_geo} gives the explicit geostrophic velocity components
\begin{equation}
\label{eq:geostrophic_velocity}
u_{\mathrm{geo}}
=
\frac{2\pi g a}{f_c}\cos(2\pi x)\sin(2\pi y),
\qquad
v_{\mathrm{geo}}
=
-\frac{2\pi g a}{f_c}\sin(2\pi x)\cos(2\pi y).
\end{equation}

The physical parameters for this test are chosen as $a=10^{-3}$, $f_c=10$ and $T=0.1$. To investigate the dynamic geostrophic adjustment process, the initial velocity is set to \eqref{eq:geostrophic_velocity}, while the initial free surface is:
\begin{equation}
\label{eq:rotating_perturbation}
\eta(x,y,0)
=
\eta_{\mathrm{geo}}(x,y)
+
\varepsilon\exp\left(-100\bigl((x-0.5)^2+(y-0.5)^2\bigr)\right),
\end{equation}
with the fixed $\varepsilon=10^{-4}$. The discrete approximations use $k=2$ and $N=16,32$.
To systematically assess the influence of the interior penalty operator on the near-balanced flow dynamics,
the penalty parameters are varied over $\beta\in\{1,3\}$ and $\sigma\in\{1,5\}$. For each reported parameter pair, errors are measured against a corresponding $N=64$ reference computation with the same polynomial degree and penalty parameters and with a time step small enough that the temporal contribution is negligible relative to the displayed spatial error.

To monitor the evolution of the near-balanced wave dynamics, the geostrophic imbalance is quantified via the $L^2$ norm:
\begin{equation}
\label{eq:geo_imbalance}
I_{\mathrm{geo}}(t)
=
\left\|
\begin{pmatrix}
-f_c v_h + g\partial_x\eta_h\\
f_c u_h + g\partial_y\eta_h
\end{pmatrix}
\right\|_{L^2(\Omega)}.
\end{equation}

The results are summarized in Table~\ref{tab5}, and the representative flow snapshots and the temporal histories of $I_{\mathrm{geo}}(t)$ are illustrated in Figure~\ref{fig3}.

\begin{table}[htbp]
\centering
\caption{Errors and geostrophic-imbalance diagnostics for the rotating benchmark.}
\label{tab5}
\begin{tabular}{cccccccc}
\toprule
$k$ & $\beta$ & $\sigma$ & $N$ & $E_{L^2}(\phi_\eta)$ & $E_{L^2}(U,V)$ & $E_{\mathrm{DG}}(U,V)$ & $I_{\mathrm{geo}}(T)$ \\
\midrule
$2$ & $1$ & $1$ & $16$ & $2.594 \times 10^{-6}$ & $1.534 \times 10^{-5}$ & $4.650 \times 10^{-3}$ & $1.369 \times 10^{-3}$ \\
 &  &  & 
 $32$ & $3.238 \times 10^{-7}$ & $1.913 \times 10^{-6}$ & $1.114 \times 10^{-3}$ & $1.295 \times 10^{-3}$\\
  &  & $5$ & $16$ & $2.596 \times 10^{-6}$ & $1.535 \times 10^{-5}$ & $5.337 \times 10^{-3}$ & $1.368 \times 10^{-3}$ \\
 &  &  & 
 $32$ & $3.240 \times 10^{-7}$ & $1.910 \times 10^{-6}$ & $1.264 \times 10^{-3}$ & $1.294 \times 10^{-3}$\\
 $2$ & $3$ & $1$ & $16$ & $2.597 \times 10^{-6}$ & $1.535 \times 10^{-5}$ & $5.337 \times 10^{-3}$ & $1.368 \times 10^{-3}$ \\
 &  &  & 
 $32$ & $3.240 \times 10^{-7}$ & $1.910 \times 10^{-6}$ & $1.264 \times 10^{-3}$ & $1.294 \times 10^{-3}$\\
  &  & $5$ & $16$ & $2.597 \times 10^{-6}$ & $1.535 \times 10^{-5}$ & $5.337 \times 10^{-3}$ & $1.368 \times 10^{-3}$ \\
 &  &  & 
 $32$ & $3.240 \times 10^{-7}$ & $1.910 \times 10^{-6}$ & $1.264 \times 10^{-3}$ & $1.294 \times 10^{-3}$\\
\bottomrule
\end{tabular}
\end{table}

\begin{figure}[!htbp]
    \centering
    \begin{subfigure}{0.3\textwidth}
        \centering
        \includegraphics[width=\linewidth]{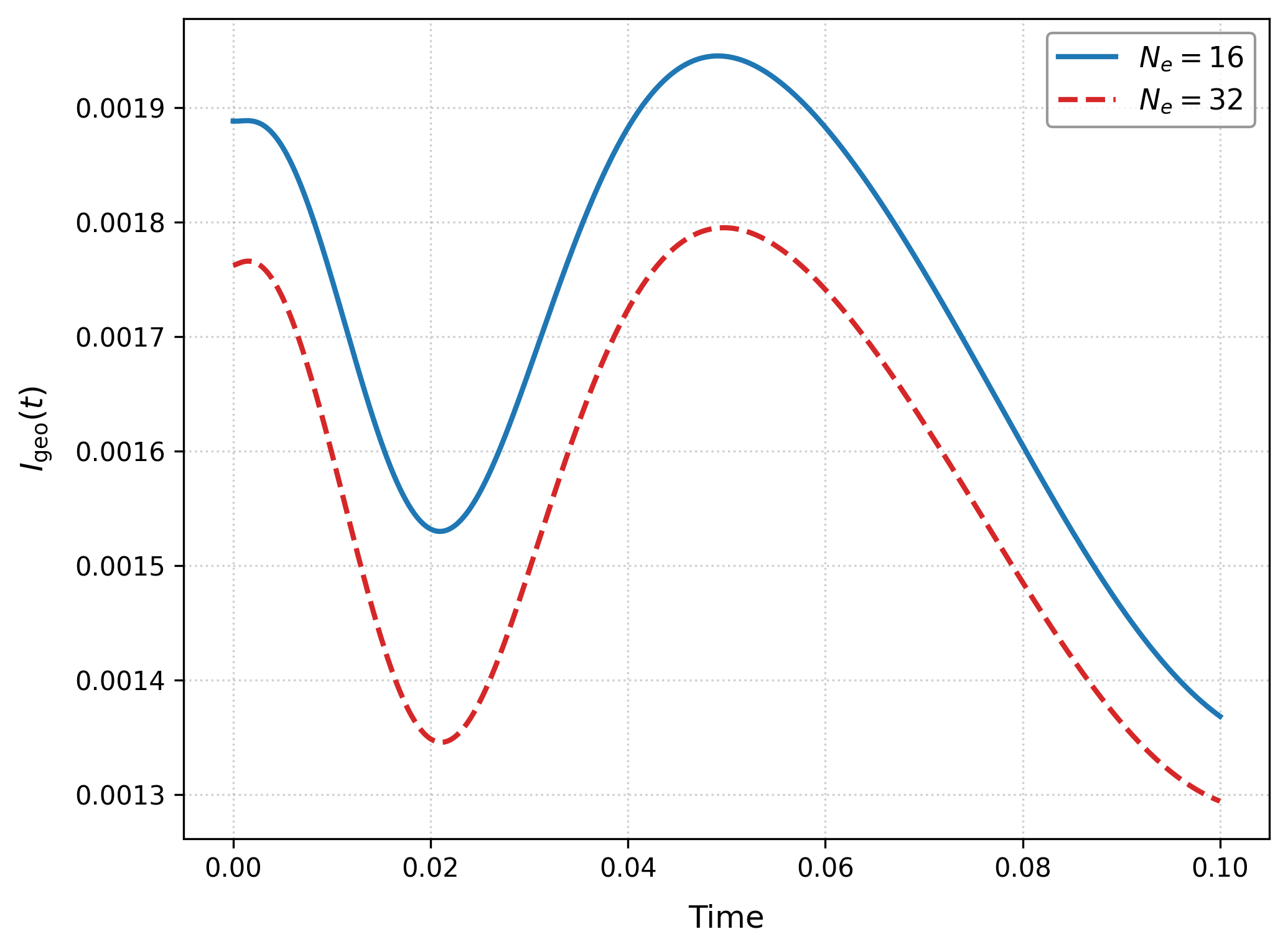}
    \end{subfigure}
    \hfill
    \begin{subfigure}{0.3\textwidth}
        \centering
        \includegraphics[width=\linewidth]{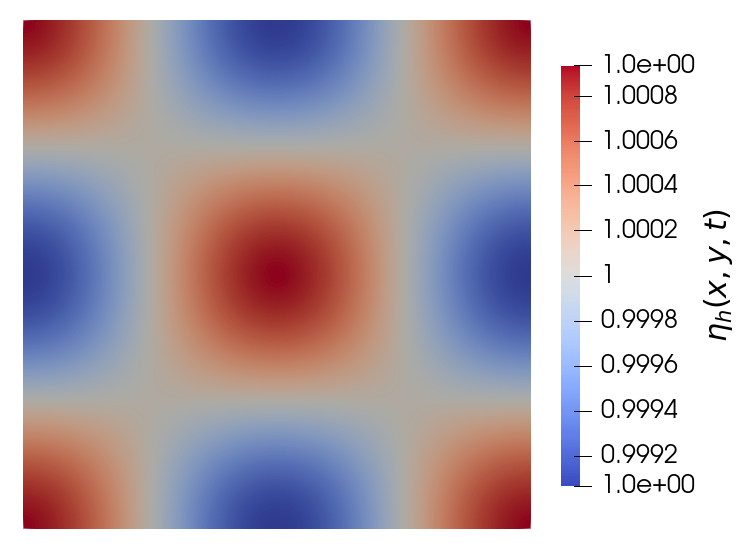}
    \end{subfigure}
    \hfill
    \begin{subfigure}{0.3\textwidth}
        \centering
        \includegraphics[width=\linewidth]{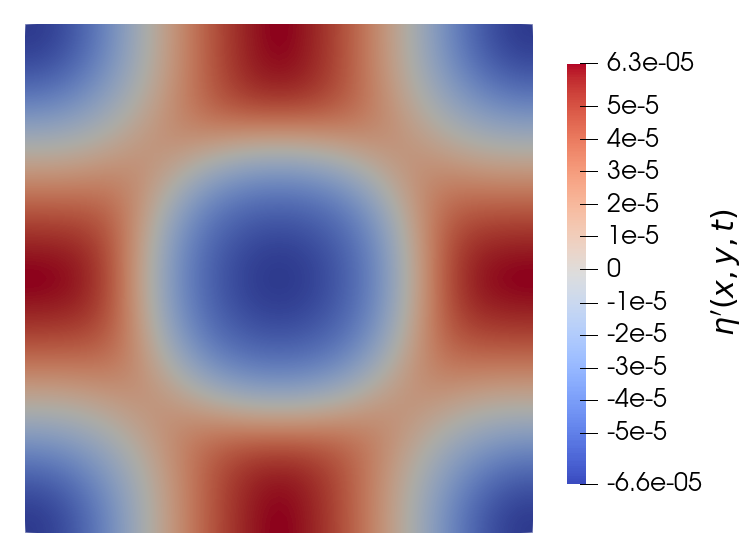}
    \end{subfigure} \\
    \vspace{0.5em}
    \begin{subfigure}{0.3\textwidth}
        \centering
        \includegraphics[width=\linewidth]{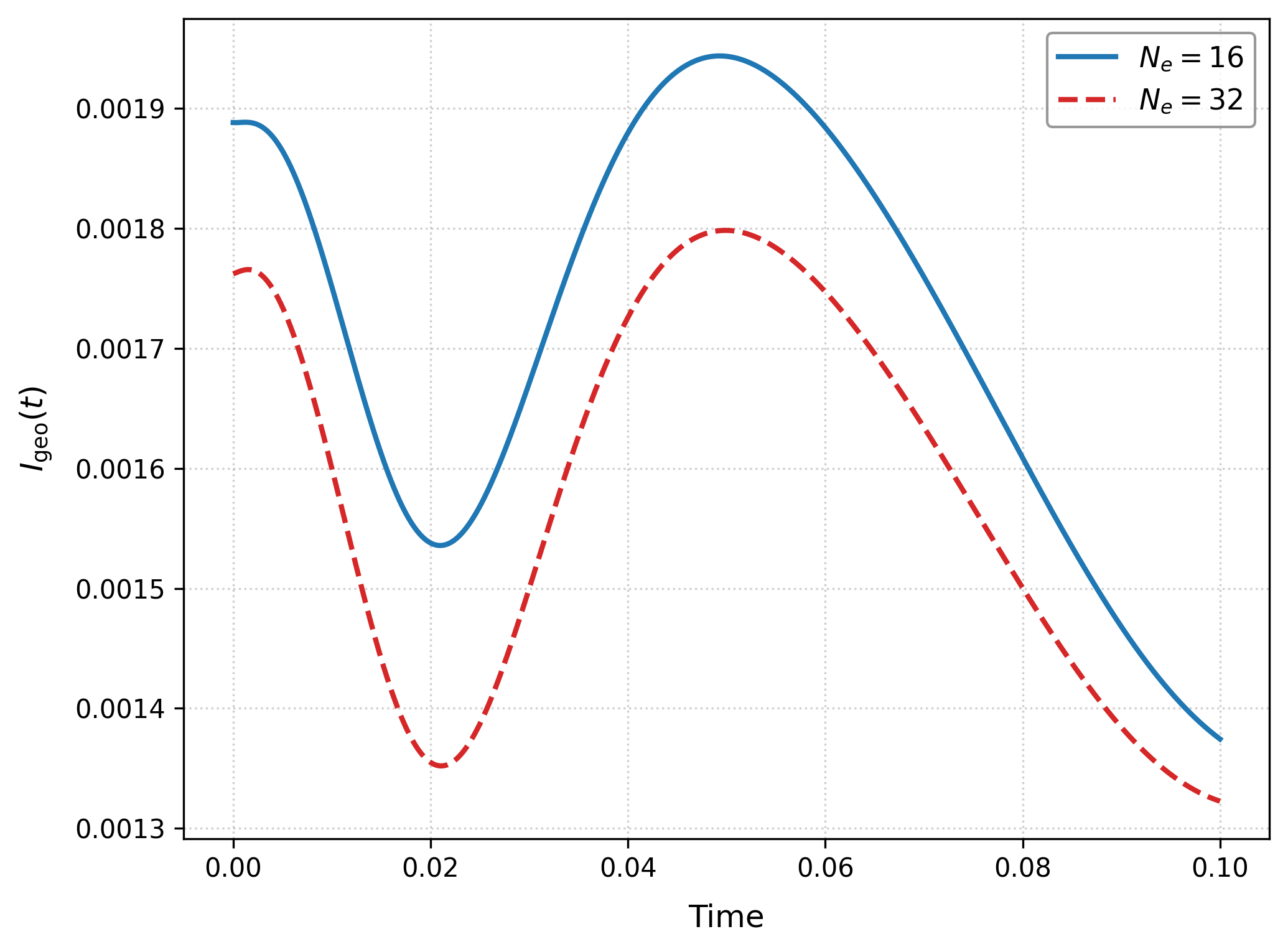}
    \end{subfigure}
    \hfill
    \begin{subfigure}{0.3\textwidth}
        \centering
        \includegraphics[width=\linewidth]{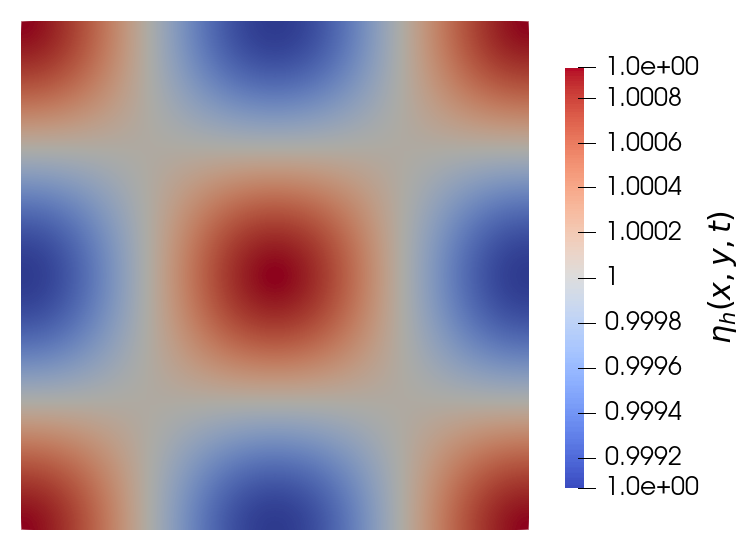}
    \end{subfigure}
    \hfill
    \begin{subfigure}{0.3\textwidth}
        \centering
        \includegraphics[width=\linewidth]{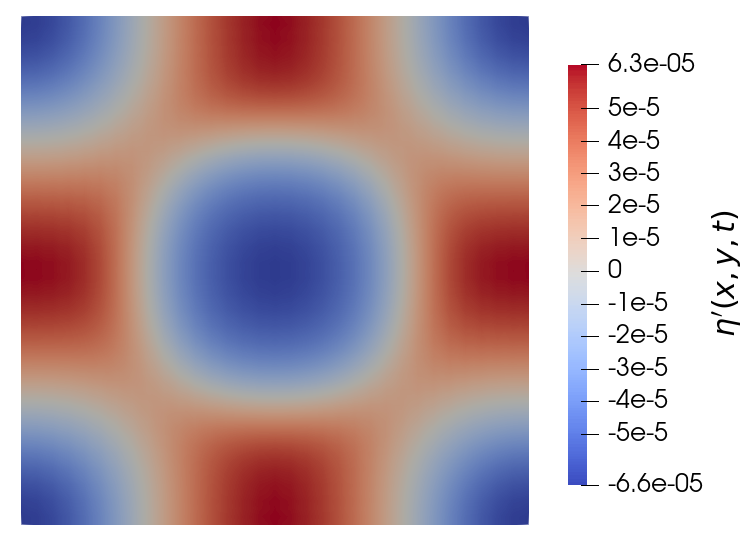}
    \end{subfigure}
    \caption{Representative numerical results for the smooth rotating benchmark with $\sigma=1.0$ and $k=2$. The top row displays results for $\beta=1.0$, while the bottom row corresponds to $\beta=3.0$. Columns from left to right present: the temporal evolution of the geostrophic imbalance $I_{\mathrm{geo}}(t)$, the total free-surface height $\eta_h$ at $t=T$ (for $N=32$), and the net free-surface height perturbation $\eta'$ at $t=T$ (for $N=32$).
    }
    \label{fig3}
\end{figure}

In contrast to the results observed in Section~\ref{sec:penalty-sensitivity}, the numerical errors remain largely insensitive to the variations in both $\beta$ and $\sigma$. As reported in Table~\ref{tab5}, for a given mesh, the $L^2$ errors for $\phi_\eta$ and $U$ ($V$) are virtually identical between $\beta = 1.0$ and $\beta = 3.0$. A similar conclusion holds true even upon mesh refinement from $N=16$ to $N=32$. 

This behavior reflects the balance-preserving property of the numerical formulation. The initial perturbation in \eqref{eq:rotating_perturbation} generates gravity waves that propagate and radiate through the domain, which is captured by  $I_{\mathrm{geo}}(t)$ in Figure~\ref{fig3}. This indicates that even under super-penalization, the scheme does not introduce artificial damping or spurious reflections of gravity waves. Furthermore, this is also confirmed by the spatial distributions of $\eta_h$ and $\eta'$ in Figure~\ref{fig3}, which demonstrate the physical consistency of the scheme near geostrophic balance.

\subsection{Topography-aware well-balanced tests}
\label{sec:topography-tests}

These tests examine the interaction between the penalty-dependent discretization and the bottom topography source term, possessing two main objectives.
First, it verifies whether the discrete pressure flux and the topography source term preserve the well-balanced rest state associated with the 
governing system \eqref{eq3}. Second, it examines the evolution of a small free-surface perturbation over non-flat topography, focusing on how 
the penalty parameters influence interface oscillations, numerical damping, positivity of the water depth, and stiffness of the resulting system.

\subsubsection{Flat-free-surface rest equilibrium}

The equilibrium is
\begin{equation}
\label{eq:lake_at_rest}
u=v=0,\qquad \eta=\eta_\star,
\end{equation}
or, in geopotential variables,
\begin{equation}
\label{eq:lake_at_rest_geo}
U=V=0,\qquad \phi_\eta=g\eta_\star,\qquad \phi=\phi_\eta+\phi_b.
\end{equation}
The bottom profile is chosen as
\begin{equation}
\label{eq:topography_profile}
b(x,y)=0.2\exp\left(-50\bigl((x-0.5)^2+(y-0.5)^2\bigr)\right),
\end{equation}
with $\eta_\star=1$. Consequently, the initial condition is given by
\begin{equation}
\label{eq:lake_initial}
\eta(x,y,0)=\eta_\star,\qquad H(x,y,0)=\eta_\star+b(x,y),\qquad u(x,y,0)=v(x,y,0)=0.
\end{equation}
The physical parameters are $f_c=0$ and $T=0.1$ with $k=2$, $N \in \{16, 32\}$, penalty exponents $\beta \in \{1, 3\}$, and prefactor $\sigma \in \{1, 5, 10\}$.

\begin{remark}
We emphasize that the equilibrium used in this test is the rest state associated with the present model, rather than a prescribed constant-depth state. Indeed, setting $u=v=0$ in the momentum equations gives $gH\nabla\eta=0$.
Since the water depth satisfies $H>0$, the admissible rest state is characterized by a flat free surface, $\eta=\eta_\star$, together with zero velocity.
Therefore, over a non-flat bottom topography $b(x,y)$, the corresponding depth is $H=\eta_\star+b(x,y)$,
which is generally not spatially constant.
A constant-depth initialization of the form $\eta+b=H_\star$ would give a non-constant $\eta$ and hence would not satisfy the steady balance of the model considered here. The purpose of the present test is therefore to verify preservation of this model-consistent rest state and to examine the behavior of small perturbations over topography.
\end{remark}

To evaluate the well-balanced property of the scheme, we monitor the free surface rest-state error $E_{\mathrm{rest}}(t)$ and the maximum spurious velocity $U_{\max}(t)$, defined respectively as
\begin{equation}
\label{eq:lake_error}
E_{\mathrm{rest}}(t) = \|\phi_{\eta,h}(t)-g\eta_\star\|_{L^2(\Omega)},
\end{equation}
\begin{equation}
\label{eq:spurious_velocity}
U_{\max}(t)= \max_{(x,y)\in\Omega} \sqrt{u_h(x,y,t)^2+v_h(x,y,t)^2}.
\end{equation}
The maximum values of these diagnostics recorded over the entire temporal trajectory $t \in [0, T]$ are documented in Table~\ref{tab6}.

\begin{table}[htbp]
\centering
\caption{Rest-state preservation errors for the flat-free-surface topography test.}
\label{tab6}
\begin{tabular}{ccccccc}
\toprule
$k$ & $\beta$ & $\sigma$ & $T$ & $N$ & $\max_t E_{\mathrm{rest}}(t)$ & $\max_t U_{\max}(t)$ \\
\midrule
 $2$ & $1$ & $1$ & $0.1$ & $16$ & $1.510 \times 10^{-14}$ & $1.191 \times 10^{-14}$ \\
 &  &  &  & $32$ & $1.490 \times 10^{-14}$ & $9.413\times 10^{-15}$ \\
$2$ & $1$ & $5$ & $0.1$ & $16$ & $1.510 \times 10^{-14}$ & $1.055 \times 10^{-14}$ \\
 &  &  &  & $32$ & $1.490 \times 10^{-14}$ & $1.037\times 10^{-14}$ \\
 $2$ & $1$ & $10$ & $0.1$ & $16$ & $1.510 \times 10^{-14}$ & $1.030 \times 10^{-14}$ \\
 &  &  &  & $32$ & $1.490 \times 10^{-14}$ & $9.581\times 10^{-15}$ \\
 $2$ & $3$ & $1$ & $0.1$ & $16$ & $1.520 \times 10^{-14}$ & $9.990 \times 10^{-15}$ \\
 &  &  &  & $32$ & $1.512 \times 10^{-14}$ & $7.229\times 10^{-15}$ \\
 $2$ & $3$ & $5$ & $0.1$ & $16$ & $1.520 \times 10^{-14}$ & $1.003 \times 10^{-14}$ \\
 &  &  &  & $32$ & $1.512 \times 10^{-14}$ & $6.776\times 10^{-15}$ \\
 $2$ & $3$ & $10$ & $0.1$ & $16$ & $1.520 \times 10^{-14}$ & $9.854 \times 10^{-15}$ \\
 &  &  & & $32$ & $1.512 \times 10^{-14}$ & $7.229 \times 10^{-15}$ \\
\bottomrule
\end{tabular}
\end{table}

The numerical results in Table~\ref{tab6} show machine-precision preservation of the tested lake-at-rest state. For all mesh and parameters, both $\max E_{\mathrm{rest}}(t)$ and $\max U_{\max}(t)$ remain at the level of machine precision, ranging from $10^{-15}$ to $10^{-14}$. This demonstrates that the discrete formulation preserves the steady state, even in the presence of the smooth non-flat bottom, thereby confirming the robustness and well-balanced nature of the proposed method.

\subsubsection{Small perturbation over topography}

To evaluate the robustness of the numerical scheme when a field interacts with localized bed variations, a localized perturbation is added to the flat-free-surface equilibrium:
\begin{equation}
\label{eq:small_perturbation}
\eta(x,y,0)=\eta_\star + \varepsilon \exp\left(-\gamma\bigl((x-x_p)^2+(y-y_p)^2\bigr)\right), \qquad H(x,y,0)=\eta(x,y,0)+b(x,y).
\end{equation}
The parameters are set to $\varepsilon=10^{-3}$, $\gamma=200$, $T=0.1$, and $(x_p,y_p)=(0.25,0.5)$. The velocity field is initially set to zero. The same mesh, polynomial degree, and penalty values used in the rest-state test are used here.

In this test, we monitor the minimum water depth
\[
H_{\min}(t)=\min_{(x,y)\in\Omega}\bigl(\eta_h(x,y,t)+b(x,y)\bigr),
\]
the perturbation amplitude
\[
E_{\mathrm{pert}}(t)=\|\phi_{\eta,h}(t)-g\eta_\star\|_{L^2(\Omega)},
\]
and the temporal maximum $J_{\max}$ of the interface-jump history defined in Section~\ref{sec:penalty-sensitivity}. Here the initial momentum is zero, so this jump maximum is not contaminated by a nonzero initial momentum projection.

These diagnostics are shown in Table~\ref{tab7}, and the snapshots with one-dimensional (1D) cross-sectional profiles are presented in Figure \ref{fig4}.

\begin{table}[htbp]
\centering
\caption{Diagnostics for the small free-surface perturbation over topography.}
\label{tab7}
\begin{tabular}{cccccccc}
\toprule
$k$ & $T$ & $\beta$ & $\sigma$ & $N$ & $H_{\min}$ & $\max_t E_{\mathrm{pert}}(t)$ & $J_{\max}$ \\
\midrule
$2$ & $0.1$ & $1$ & $1$ & $16$ & $9.999 \times 10^{-1}$ & $8.693 \times 10^{-4}$ & $4.505 \times 10^{-4}$\\
 &  &  &  & $32$ & $9.999 \times 10^{-1}$ & $8.694 \times 10^{-4}$ & $9.516 \times 10^{-5}$\\
$2$ & $0.1$ & $1$ & $5$ & $16$ & $9.999 \times 10^{-1}$ & $8.693 \times 10^{-4}$ & $4.031 \times 10^{-4}$\\
 &  &  &  & $32$ & $9.999 \times 10^{-1}$ & $8.694 \times 10^{-4}$ & $8.646 \times 10^{-5}$\\
$2$ & $0.1$ & $1$ & $10$ & $16$ & $9.999 \times 10^{-1}$ & $8.693 \times 10^{-4}$ & $3.585 \times 10^{-4}$\\
 &  &  &  & $32$ & $9.999 \times 10^{-1}$ & $8.694 \times 10^{-4}$ & $7.898 \times 10^{-5}$\\
$2$ & $0.1$ & $3$ & $1$ & $16$ & $9.999 \times 10^{-1}$ & $8.693 \times 10^{-4}$ & $1.176 \times 10^{-4}$\\
 &  &  &  & $32$ & $9.999 \times 10^{-1}$ & $8.694 \times 10^{-4}$ & $1.156 \times 10^{-5}$\\
$2$ & $0.1$ & $3$ & $5$ & $16$ & $9.999 \times 10^{-1}$ & $8.693 \times 10^{-4}$ & $3.470 \times 10^{-5}$\\
 &  &  &  & $32$ & $9.999 \times 10^{-1}$ & $8.694 \times 10^{-4}$ & $2.713 \times 10^{-6}$\\
$2$ & $0.1$ & $3$ & $10$ & $16$ & $9.999 \times 10^{-1}$ & $8.693 \times 10^{-4}$ & $1.874 \times 10^{-5}$\\
 &  &  &  & $32$ & $9.999 \times 10^{-1}$ & $8.694 \times 10^{-4}$ & $1.389 \times 10^{-6}$\\
\bottomrule
\end{tabular}
\end{table}

\begin{figure}[!htbp]
\centering

\begin{subfigure}[b]{0.24\textwidth}
    \centering
    \includegraphics[width=\textwidth]{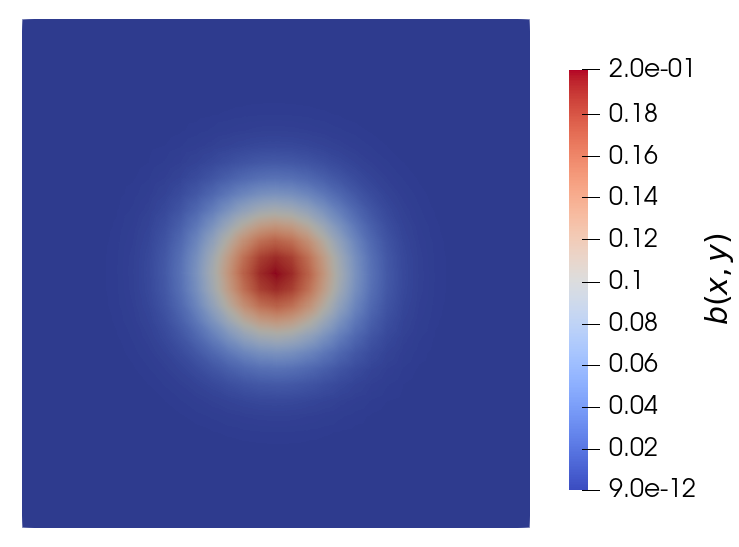}
\end{subfigure}
\hfill
\begin{subfigure}[b]{0.24\textwidth}
    \centering
    \includegraphics[width=\textwidth]{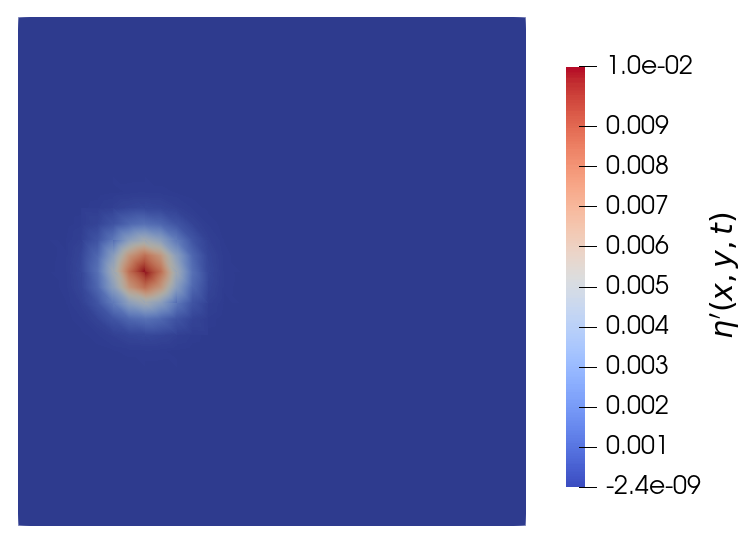}
\end{subfigure}
\hfill
\begin{subfigure}[b]{0.24\textwidth}
    \centering
    \includegraphics[width=\textwidth]{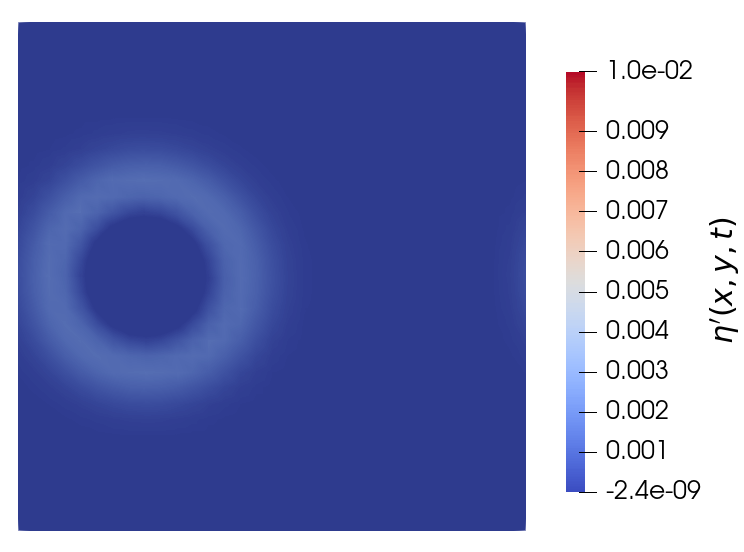}
\end{subfigure}
\hfill
\begin{subfigure}[b]{0.24\textwidth}
    \centering
    \includegraphics[width=\textwidth]{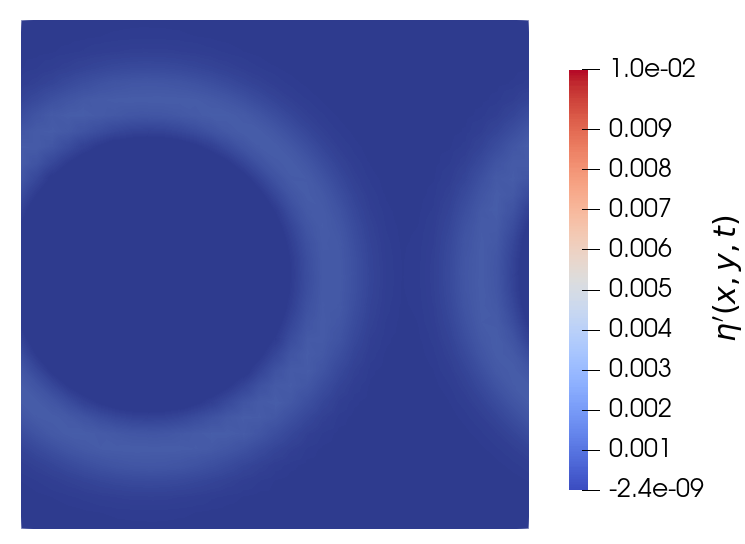}
\end{subfigure}

\vspace{0.3cm}

\begin{subfigure}[b]{0.24\textwidth}
    \centering
    \includegraphics[width=\textwidth]{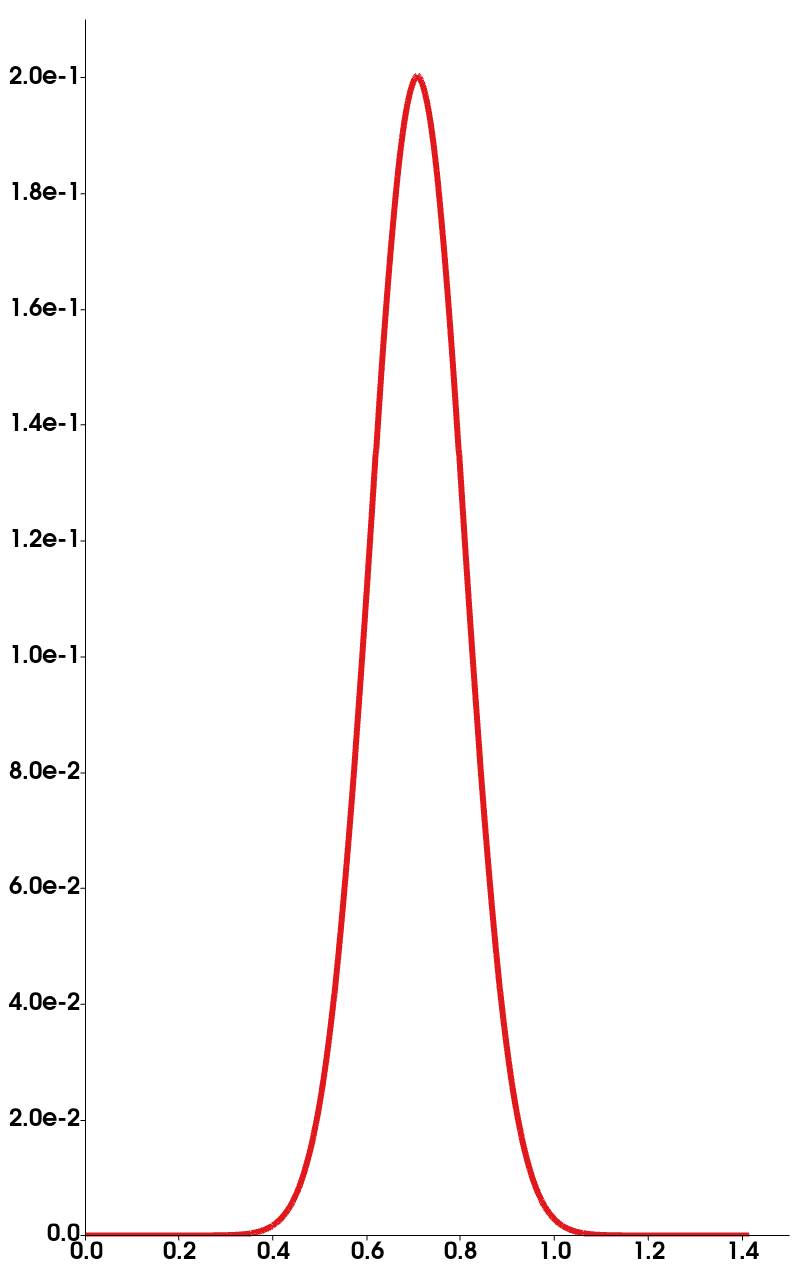}
    \caption{Bottom \\ ($\beta=1$)}
\end{subfigure}
\hfill
\begin{subfigure}[b]{0.24\textwidth}
    \centering
    \includegraphics[width=\textwidth]{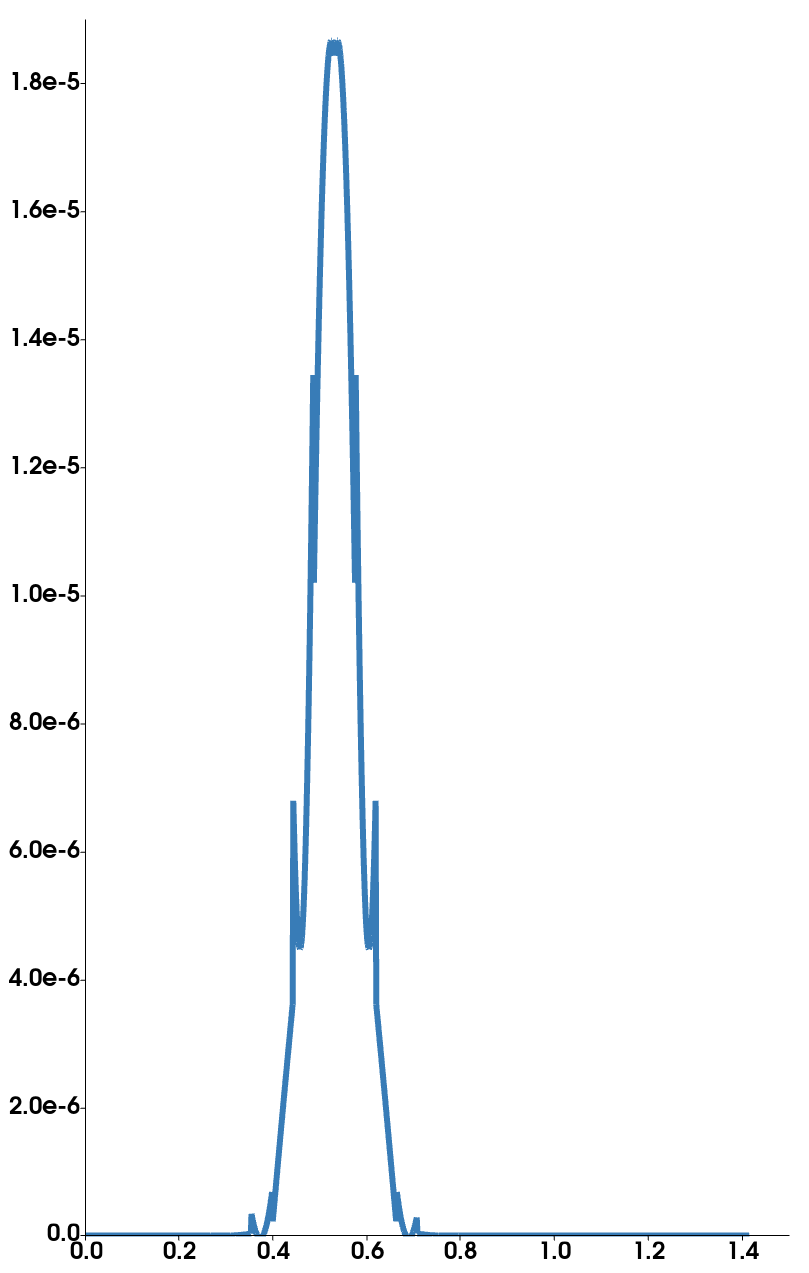}
    \caption{$\eta^\prime$ at $t=0$ \\ ($\beta=1$)}
\end{subfigure}
\hfill
\begin{subfigure}[b]{0.24\textwidth}
    \centering
    \includegraphics[width=\textwidth]{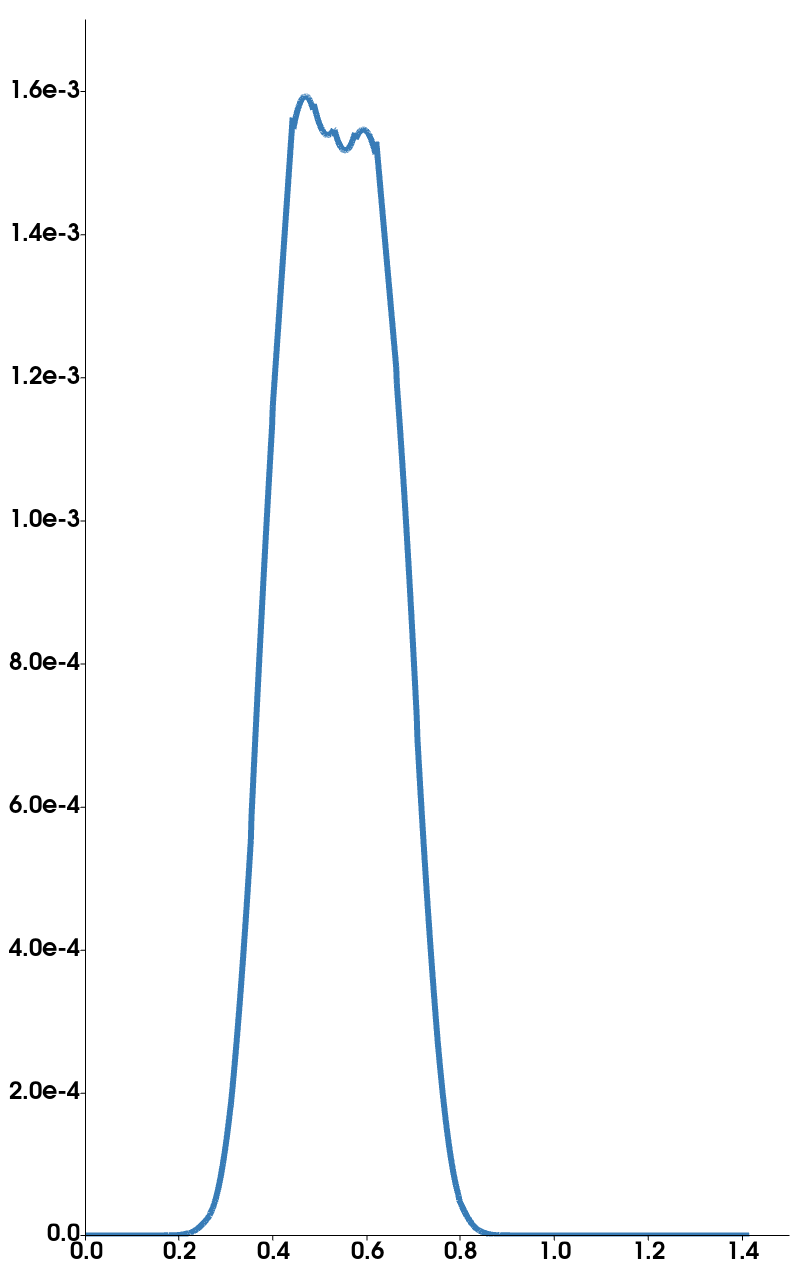}
    \caption{$\eta^\prime$ at $t=0.05$ \\ ($\beta=1$)}
\end{subfigure}
\hfill
\begin{subfigure}[b]{0.24\textwidth}
    \centering
    \includegraphics[width=\textwidth]{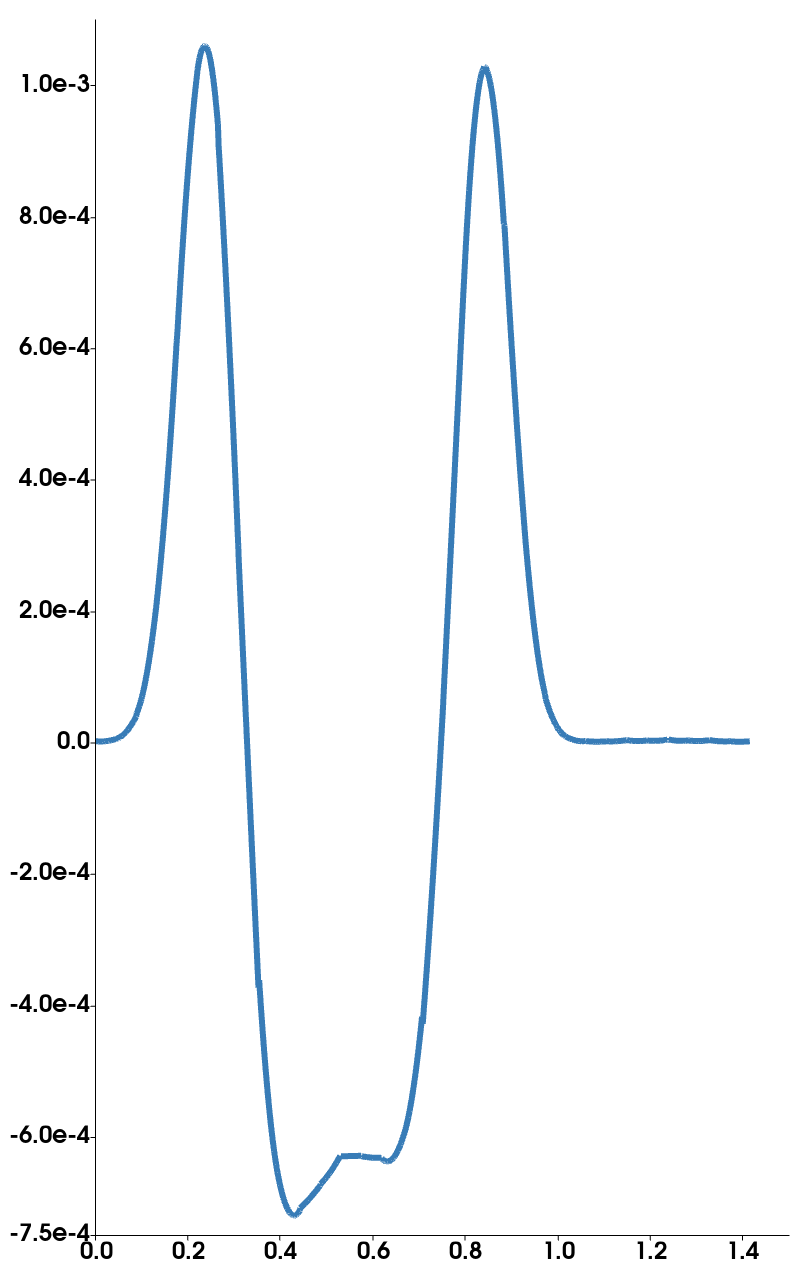}
    \caption{$\eta^\prime$ at $t=0.1$ \\ ($\beta=1$)}
\end{subfigure}

\vspace{0.3cm} 

\begin{subfigure}[b]{0.24\textwidth}
    \centering
    \includegraphics[width=\textwidth]{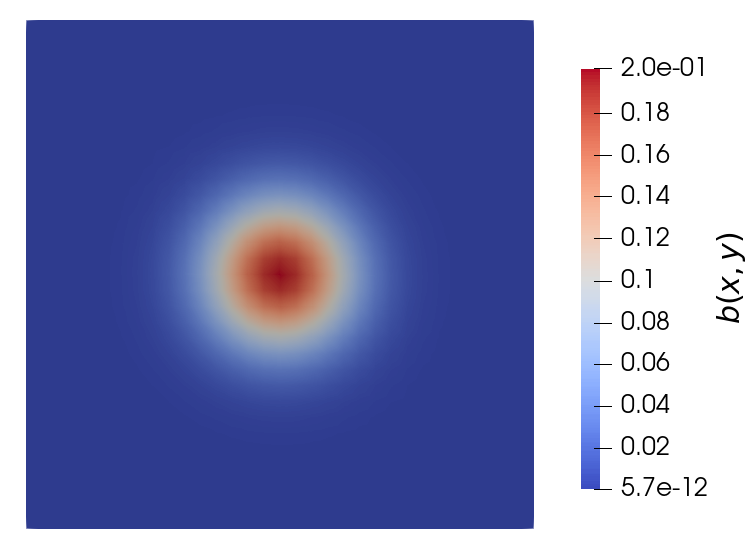}
\end{subfigure}
\hfill
\begin{subfigure}[b]{0.24\textwidth}
    \centering
    \includegraphics[width=\textwidth]{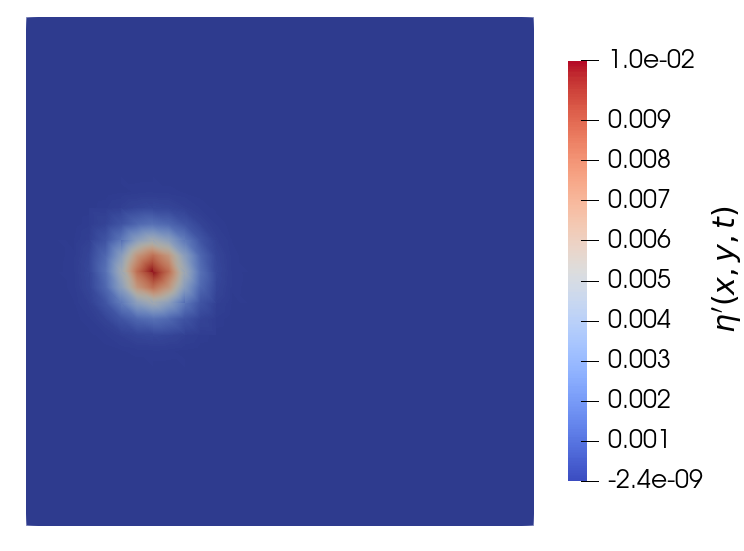}
\end{subfigure}
\hfill
\begin{subfigure}[b]{0.24\textwidth}
    \centering
    \includegraphics[width=\textwidth]{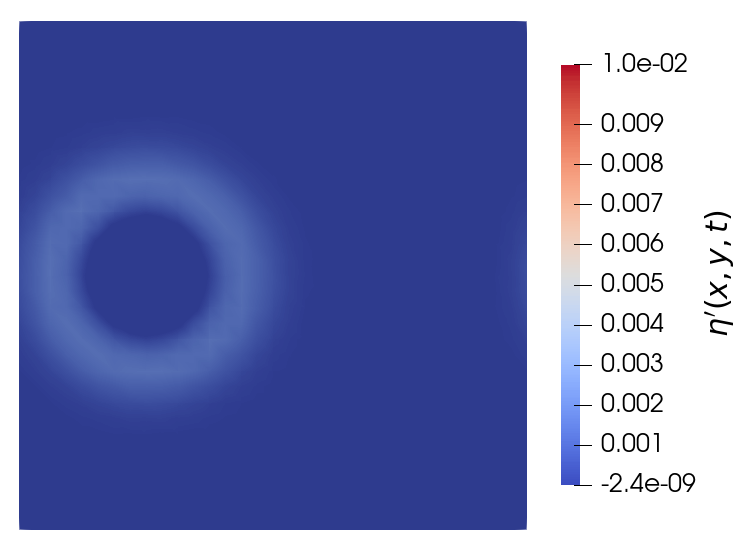}
\end{subfigure}
\hfill
\begin{subfigure}[b]{0.24\textwidth}
    \centering
    \includegraphics[width=\textwidth]{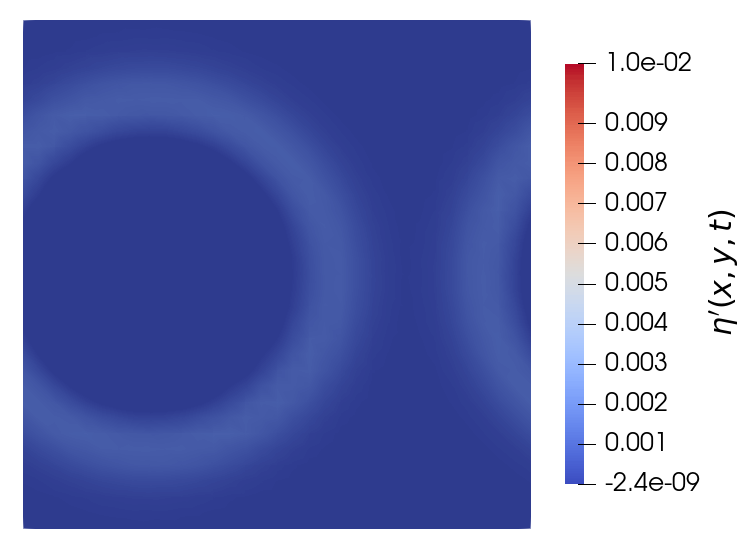}
\end{subfigure}

\vspace{0.5cm}

\begin{subfigure}[b]{0.24\textwidth}
    \centering
    \includegraphics[width=\textwidth]{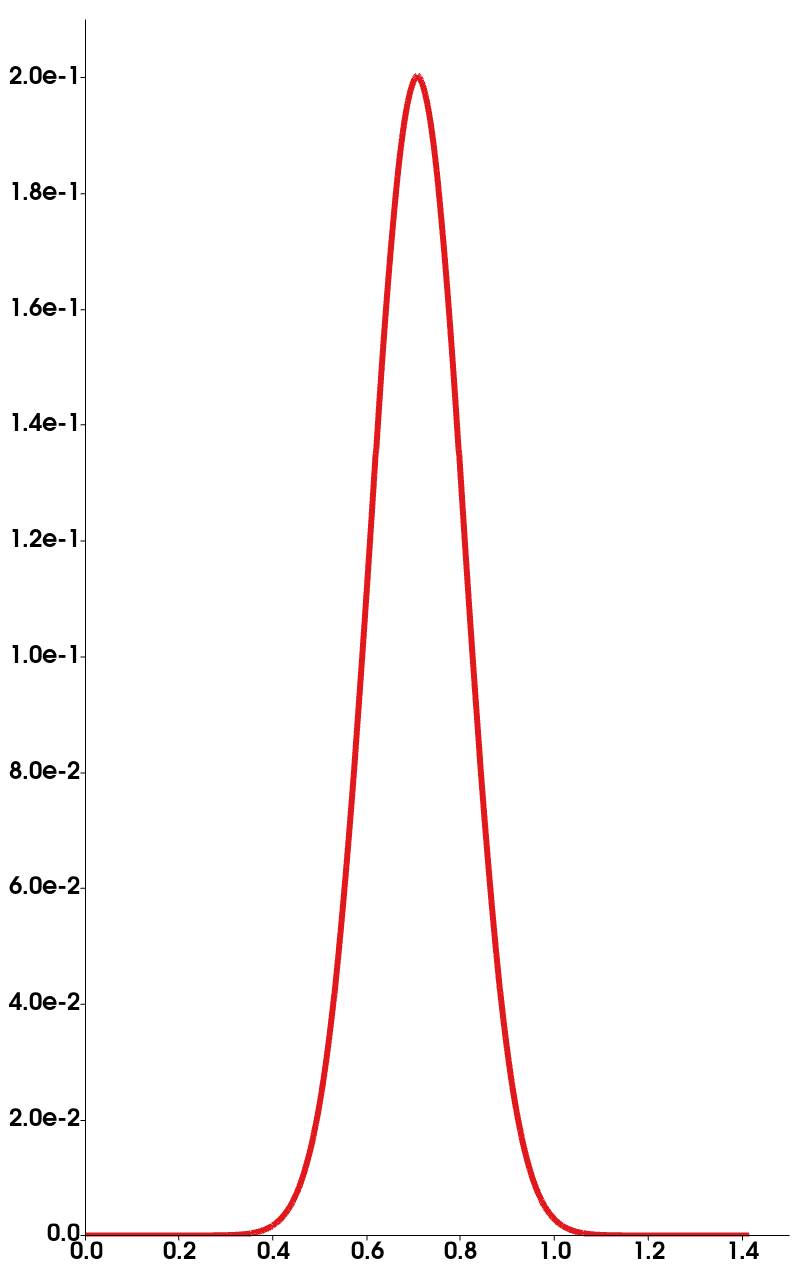}
    \caption{Bottom \\ ($\beta=3$)}
\end{subfigure}
\hfill
\begin{subfigure}[b]{0.24\textwidth}
    \centering
    \includegraphics[width=\textwidth]{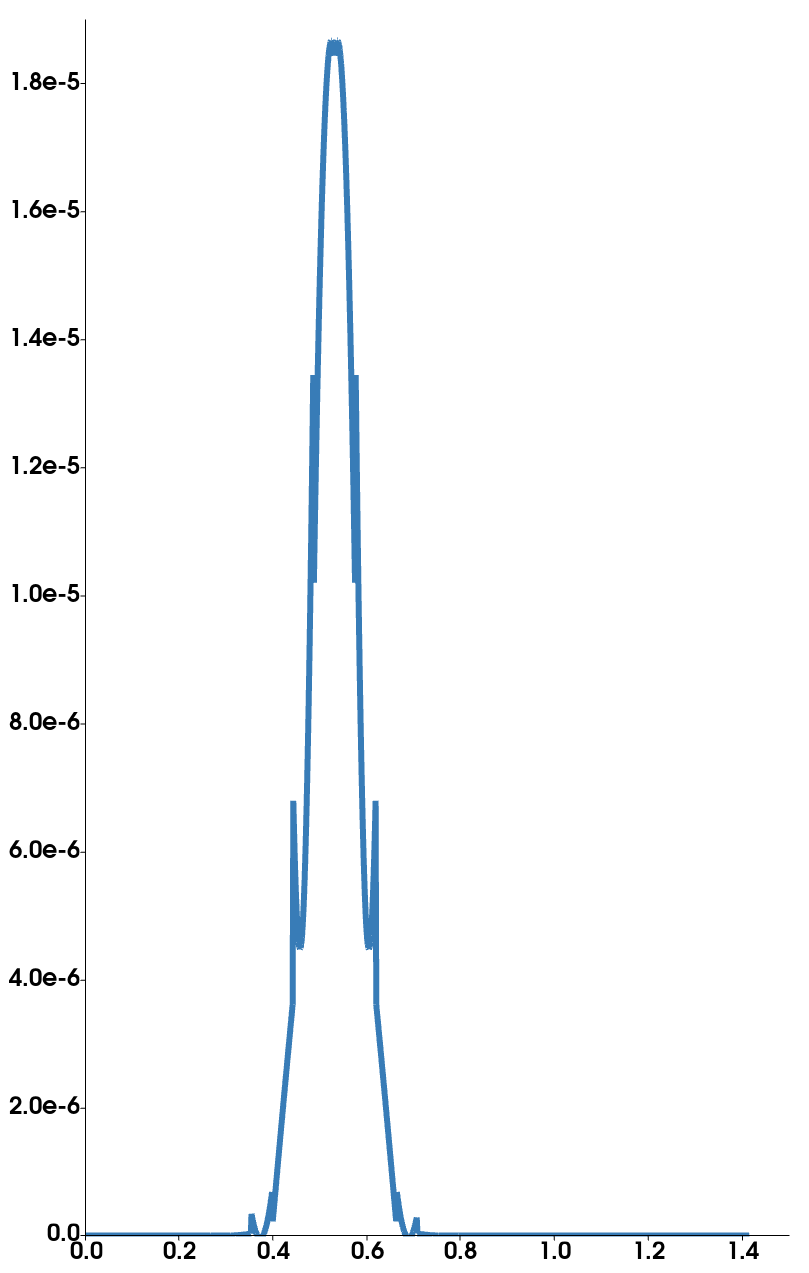}
    \caption{$\eta^\prime$ at $t=0$ \\ ($\beta=3$)}
\end{subfigure}
\hfill
\begin{subfigure}[b]{0.24\textwidth}
    \centering
    \includegraphics[width=\textwidth]{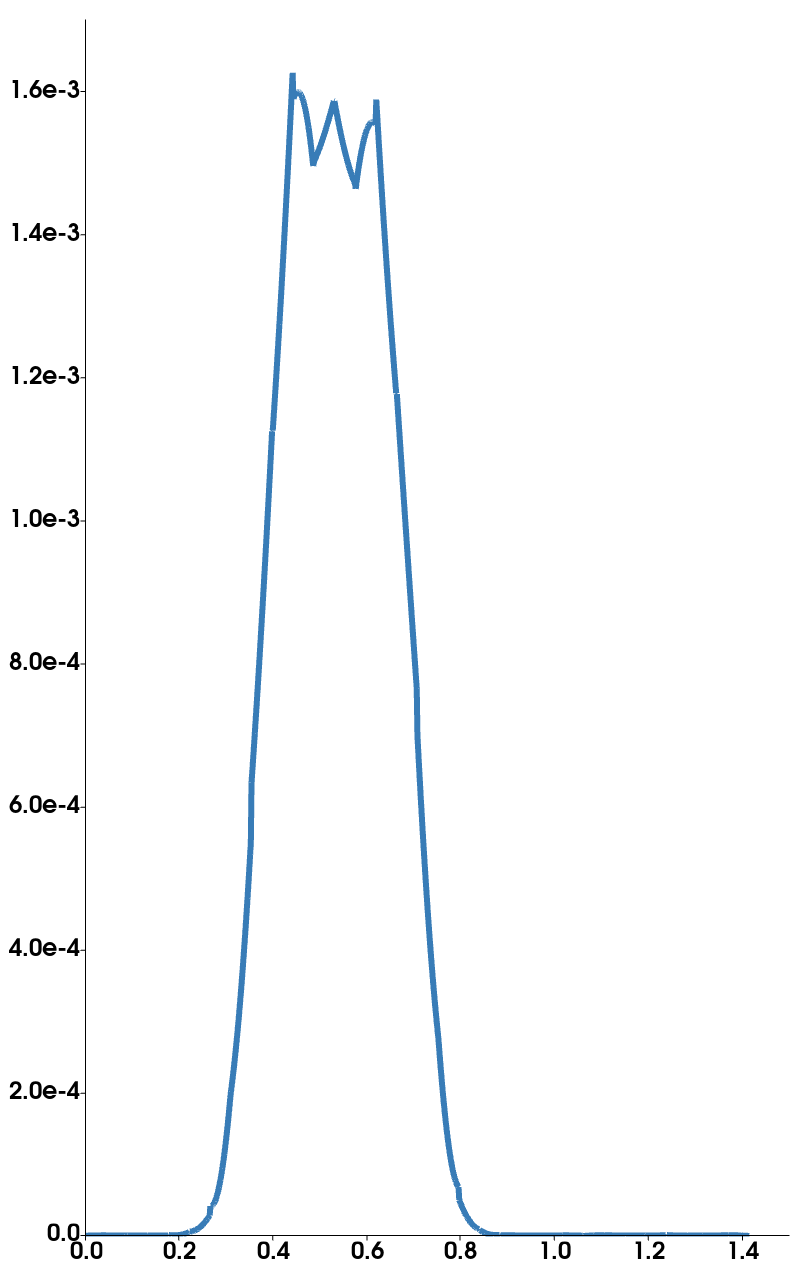}
    \caption{$\eta^\prime$ at $t=0.05$ \\ ($\beta=3$)}
\end{subfigure}
\hfill
\begin{subfigure}[b]{0.24\textwidth}
    \centering
    \includegraphics[width=\textwidth]{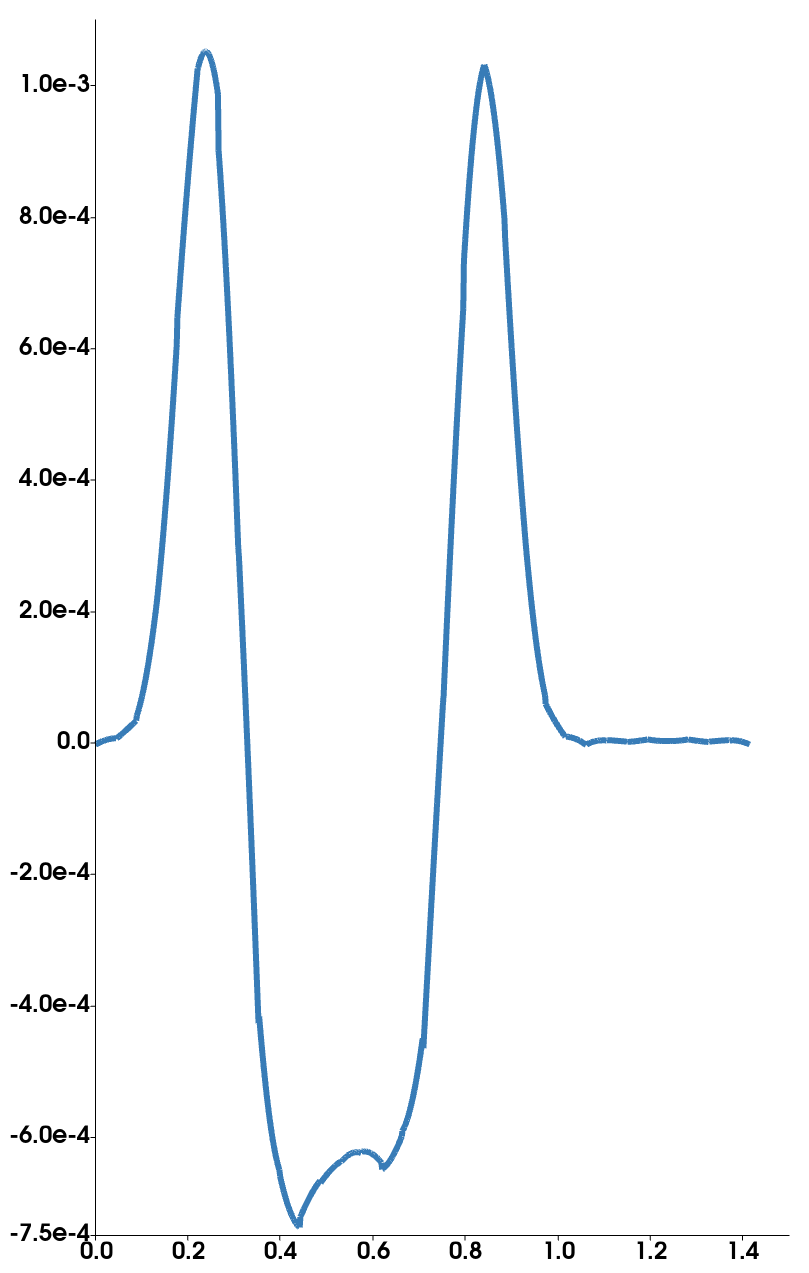}
    \caption{$\eta^\prime$ at $t=0.1$ \\ ($\beta=3$)}
\end{subfigure}

\caption{Representative solution snapshots and corresponding 1D cross--sectional profiles extracted along the main diagonal $y=x$ with $N=16$ and $\sigma=5$. Rows 1--2 correspond to the case $\beta=1$, whereas Rows 3--4 correspond to $\beta=3$.}
\label{fig4}
\end{figure}

As shown in Table~\ref{tab7}, the maximum perturbation amplitude $\max_t E_{\text{pert}}(t)$ remains virtually invariant over the short term when switching from $\beta=1$ to $\beta=3$. This numerical constancy rigorously proves that the super-penalization under $\beta=3$ avoids introducing excessive dissipation to the underlying wave propagation. Furthermore, the minimum total water depth $H_{\min}$ remains strictly positive, inherently preserving the positivity and highlighting the  robustness of the numerical operator.

However, a distinct divergence is observed in the behavior of the interface jumps. The super-penalized choice $\beta=3$ achieves a significantly sharper jump control, reflected in reducing the maximum jump seminorm $J_{\max}$ compared to $\beta=1$ (e.g., dropping from $8.646 \times 10^{-5}$ to $2.713 \times 10^{-6}$ for $N=32$ and $\sigma=5$). 

Visually, as shown in Figure~\ref{fig4}, the wave profiles under $\beta=1$ and $\beta=3$ are almost indistinguishable, demonstrating that both methods accurately track the perturbation wave. The only subtle divergence is the appearance of some localized grid-scale oscillations near the cell boundaries under $\beta=3$. This is a typical and well-understood phenomenon associated with over-penalization; crucially, this minor variation remains stable and does not amplify over extended simulation times.

Ultimately, this experiment highlights a critical computational trade-off that must be stated explicitly: While the super-penalized choice $\beta=3$ delivers superior mathematical control over interface jumps ($J_{\max}$), it does so at the expense of computational efficiency. The severe algebraic stiffness injected reduces the admissible explicit time step size, significantly increasing the overall CPU time. Therefore, for geophysical applications, the standard scaling $\beta=1$ remains the more balanced and computationally efficient choice.

\section{Conclusion}
\label{sec:conclusion}

We studied NIPG momentum diffusion in a DG discretization of the viscous rotating shallow-water equations in geopotential variables. The formulation combines a local Lax--Friedrichs treatment of the hyperbolic flux with an NIPG discretization of the viscous operator and uses the penalty law $\mu_e=\sigma h_e^{-\beta}$ to isolate the influence of interface stabilization. The analytical framework identifies the penalty-dependent jump scaling as the key mechanism entering continuity, coercivity, stability, and error estimates, while the numerical section is organized to test the same mechanism across convergence studies, parameter scans, rotating benchmarks, and topography-aware balance tests.

Our numerical experiments validate the theoretical predictions and further reveal an asymmetric sensitivity between the geopotential and momentum fields under different penalty configurations. In particular, under super-penalization ($\beta=3$), the momentum field still attains the optimal $L^2$ convergence rates in the convection-dominated regime tested here. However, this choice substantially increases the numerical stiffness of the system, leading to more restrictive time-step constraints and higher computational costs.
From a practical standpoint, these results suggest that the standard penalty scaling ($\beta=1$), combined with a moderate prefactor, provides a more balanced and efficient choice for simulating viscous shallow water flows.

For future investigation, extending the present single-layer formulation to multi-layer shallow water systems is of interest for capturing more complex baroclinic dynamics \cite{Gahounzo2026}. In addition, the strong stiffness introduced by the NIPG penalty terms motivates the development of semi-implicit time-integration methods, where the stiff viscous penalty operators and fast gravity-wave components are treated implicitly while the nonlinear advection terms remain explicit. Such approaches will require efficient solvers and robust preconditioners, such as multigrid-based \cite{Betteridge2021} or block preconditioners \cite{Cotter2023}, to handle the resulting large coupled linear systems efficiently.

\backmatter

\section*{Declarations}

\paragraph{Funding}
This work was supported by the Innovation Research Foundation of the National University of Defense Technology, the Youth Elite Scientists Sponsorship Program by CAST, and the National Natural Science Foundation of China (grant no.~12371374).

\paragraph{Competing interests}
The authors have no relevant financial or non-financial interests to disclose.

\paragraph{Ethics approval}
Not applicable.

\paragraph{Consent to participate}
Not applicable.

\paragraph{Consent for publication}
Not applicable.

\paragraph{Data and code availability}
All numerical data reported in the tables are contained in the manuscript. The source code used to generate the numerical results are available in \url{https://github.com/linlin-cabbage/Manuscript_2026}.

\paragraph{Author contributions}
Xue Zhang: conceptualization, methodology, software, validation, investigation, visualization, and writing--original draft. Jingmin Xia: conceptualization, methodology, supervision, project administration, funding acquisition, and writing--review and editing. Xu Qian: supervision, project administration, funding acquisition, and writing--review and editing. All authors read and approved the final manuscript.

\bibliography{main}

\end{document}